%% file: main.tex
\documentclass[11pt]{article}
\usepackage[margin=1.2in]{geometry}
\usepackage{amsmath,amsfonts,amssymb,amsthm}
\usepackage[
colorlinks,
citecolor=blue,
urlcolor=blue,
linkcolor=red, 
backref=page,
linktocpage=true]{hyperref}
\usepackage[numbers]{natbib}
\usepackage{graphicx}   
\usepackage{caption}    
\usepackage{subcaption} 
\usepackage{xcolor}
\usepackage{authblk}
\usepackage{tocloft}
\usepackage{nicefrac}
\numberwithin{equation}{section}

\newcommand{\eps}{\epsilon}
\newcommand{\E}{\mathbb{E}}

\newcommand{\la}{\langle}
\newcommand{\ra}{\rangle}
\renewcommand{\d}{\textnormal{d}}
\newcommand{\kl}{D_{\textsc{kl}}}
\newcommand{\Tr}{\text{Tr}}
\newcommand{\dsphere}{\mathbb{S}^{d-1}}
\renewcommand{\Re}{\mathbb{R}}
\newcommand{\hatV}{\widehat{V}}
\renewcommand{\subset}{\subseteq}

\newcommand{\dball}{\mathbb{B}^{d}}
\newcommand{\op}{{\textnormal{op}}}
\newcommand{\muhat}{\widehat{\mu}}

\def\ddefloop#1{\ifx\ddefloop#1\else\ddef{#1}\expandafter\ddefloop\fi}
\def\ddef#1{\expandafter\def\csname cal#1\endcsname{\ensuremath{\mathcal{#1}}}}
\ddefloop ABCDEFGHIJKLMNOPQRSTUVWXYZ\ddefloop

\theoremstyle{plain}
\newtheorem{theorem}{Theorem}[section]
\newtheorem{proposition}[theorem]{Proposition}

\newtheorem{corollary}[theorem]{Corollary}
\newtheorem{lemma}[theorem]{Lemma}

\theoremstyle{definition}
\newtheorem{example}[theorem]{Example}
\newtheorem{remark}[theorem]{Remark}
\newtheorem{definition}[theorem]{Definition}

\renewcommand{\citealt}{\citet}
\renewcommand{\citealp}{\citep}

\newif\ifarxiv
\arxivtrue      

\title{A variational approach to dimension-free \\
self-normalized concentration\thanks{Published in \emph{Transactions on Machine Learning Research}.}}
\author[1]{Ben Chugg}
\author[2]{Aaditya Ramdas}
\affil[1]{Departments of Machine Learning and Statistics, CMU}
\affil[2]{Department of Statistics, Stanford University}
\affil[ ]{\texttt{benchugg@cmu.edu; aramdas@stanford.edu}}
\date{\today}

\begin{document}

\maketitle

\input{abstract}

{\small
\tableofcontents
}

\input{body}

\subsection*{Acknowledgments}
BC was supported in part by NSERC-PGSD, grant no.
567944. Both authors acknowledge support from NSF grants IIS-2229881 and
DMS-2310718. The authors thank the reviewers for their careful reading of the
manuscript, for identifying several errors, and for suggestions that strengthened
the results.

{\small 
\bibliographystyle{plainnat}
\bibliography{main}
}

\appendix

\input{appendix}
\end{document}

%% file: abstract.tex
\begin{abstract}
    We study the self-normalized concentration of vector-valued stochastic processes.  
    We focus on bounds for ``sub-$\psi$'' processes, a well-known and quite general class  that encompasses a wide variety of well-known tail conditions (including sub-exponential, sub-Gaussian, sub-gamma, sub-Poisson, and several heavy-tailed settings without a moment generating function, such as conditionally symmetric observations or observations with only finite second or third moments).
    Our results recover and generalize the influential bound of \citet{pena2004self} (proved again in \citealt{abbasi2011improved}) in the sub-Gaussian case. Further, we fill a gap in the literature between determinant-based bounds and more recent bounds based on condition numbers. 
    As applications we prove a Bernstein inequality for random vectors satisfying a moment condition (a more general condition than boundedness), and also provide the first dimension-free self-normalized  empirical Bernstein inequality. Our techniques are based on the variational (PAC-Bayes) approach to concentration.  
\end{abstract}

%% file: body.tex
\section{Introduction}
The modern theory of finite-sample, self-normalized concentration originates primarily in the work of de la Pe\~{n}a, Klass, Lai, Shao and co-authors in the 2000s~\citep{pena2004self,pena2007pseudo,pena2008self,pena2009exponential,pena2009theory}. Much of this work was focused on scalar-valued processes until \citet{pena2009theory} began studying the concentration of \emph{vector-valued} self-normalized processes in Euclidean spaces. These results take the form of bounds on $\|S_t\|_{V_t^{-1}}$ where $(S_t)_{t\geq 0}$ is a stochastic process in $\Re^d$ and $(V_t)_{t\geq 0}$, $V_t\in \Re^{d\times d}$, is some associated positive-definite variance process.
In the sub-Gaussian case, \citet{pena2004self} provide bounds---much earlier than the oft-cited bound of \citet{abbasi2011improved}---depending on the log-determinant of $V_t$ using the method of mixtures with a multivariate Gaussian mixing distribution.
Under slightly more general tail assumptions on $(S_t)$, \citet{pena2009theory} provide bounds that $\|S_t\|_{V_t^{-1}}$ lies in a convex set, though these are not computable in closed-form.

What about more general classes of processes? Recently, \citet{whitehouse2023time} gave self-normalized concentration bounds for vector-valued sub-$\psi$ processes, a powerful and general abstraction first introduced by \citet{howard2020time} 
which captures many distributions of interest, including sub-Poisson, sub-exponential, sub-gamma, in addition to several heavy tailed settings where the moment generating function may not exist (see Section~\ref{sec:prelims}). While comprehensive, the bounds of \citet{whitehouse2023time} deviate from the aforementioned log-determinant bounds in two ways. First, they are dimension-dependent and second, they scale with the condition number of $V_t$ instead of the determinant.
Neither dependency is strictly better than the other; both can be tighter in different regimes. However, the discrepancy  invites the following question:
\begin{quote}
\emph{Do dimension-free, determinant-based bounds exist for general sub-$\psi$ processes?}
\end{quote}

This paper answers in the affirmative, providing dimension-free, determinant-based self-normalized concentration inequalities for vector-valued sub-$\psi$ processes.
\ifarxiv
To demonstrate the scope of our technique,
in Section~\ref{sec:sub-gaussian} we will recover the bound of~\citet{abbasi2011improved} exactly and show that it holds without modification for general sub-Gaussian processes (i.e., outside of their particular bandit setting).
We then turn our attention to general sub-$\psi$ processes.
Here our bounds have a similar functional form to the sub-Gaussian bound but are modified by multiplicative and additive factors that depend on $\psi$. Nevertheless, they remain closed-form and easily computable.

\fi
Our techniques are based on the variational approach to concentration (sometimes called the PAC-Bayesian approach~\citep{catoni2017dimension,catoni2018dimension}), which has been successfully leveraged in recent years to study the concentration of random vectors and matrices~\citep{zhivotovskiy2024dimension,oliveira2016lower,nakakita2024dimension,ziemann2025vector,chugg2025time}. The variational approach is comparable to a data-dependent method of mixtures,
in which we choose the mixture distribution to be data-dependent but pay the price of this dependency with an $f$-divergence.

While closing the gap between determinant-based and condition number-based bounds is mathematically interesting, it is also practically relevant. Beyond their application in bandit problems~\citep{abbasi2011improved,abbasi2011online,vial2022improved,chowdhury2017kernelized,lattimore2020bandit,whitehouse2023sublinear}, self-normalized bounds are used in many areas, such as system identification~\citep{matni2019tutorial}, control of dynamic systems~\citep{lai1982least}, estimation in autoregressive processes and diffusion-limited aggregation processes~\citep{bercu2019new},  sequential inference in time-series~\citep{shao2010self}, and Markov decision-processes~\citep{zhang2024achieving,vial2022improved}. Moreover, in many of these applications, matrices are often ill-conditioned due to numerical instability or anisotropy~\citep{higham2002accuracy}. Determinant-based bounds which omit the dependence on condition numbers can be useful in such areas.

\subsection{History \& Related Work}
\label{sec:related-work}

The canonical example of a self-normalized process is the $t$-statistic, introduced by William Gosset (aka Student) in 1908~\citep{student1908probable}. The behavior of the $t$-statistic can be understood by studying objects of the form $\sum_{i\leq n} X_i / \sqrt{\sum_{i\leq n} X_i^2}$, where $X_1,\dots,X_n$ are iid scalar observations. The asymptotic behavior of such ``self-normalized'' sums was the focus of a significant body of work beginning in the late 60s and early 70s~\citep{efron1969student,logan1973limit} and gaining significant momentum in the 90s~\citep{griffin1989self,griffin1991some, jing1999exponential,gine1997student,gine1998lil,shao1997self}.

De la Pe\~{n}a and others then built off of this work to provide the foundations of a general theory of self-normalized concentration, best summarized in the 2008 book of de la Pe\~{n}a, Lai, and Shao~\citep{pena2008self}; the book is mostly focused on the scalar setting (with some sub-Gaussian vector bounds as we describe in the following paragraph).
In the scalar setting, subsequent work by \citet{bercu2008exponential,bercu2019new} extended these results to heavy-tailed increments, developing self-normalized bounds when the normalization process $(V_t)$ is the predictable or total quadratic variation. In 2020, \citet{howard2020time} further generalized this framework by studying time-uniform concentration of sub-$\psi$ processes---processes characterized by specific tail conditions on the cumulant generating functions---thereby providing a general, unifying framework for both self-normalized and non-self-normalized scalar (and matrix) bounds.

Beyond the scalar setting, \citet{pena2009theory} (see also \citet[Chapter 14]{pena2008self}) considered the self-normalized concentration of vector-valued processes, which take the form of bounds on $\|S_t\|_{V_t^{-1}}$ for a vector-valued process $(S_t)$ and a matrix-valued process $(V_t)$. Hotelling's $T^2$-statistic~\citep{hotelling1931generalization}---the generalization of the Student t-statistic to the multivariate setting---is an example of such a process. \citet{pena2009theory} give bounds on the probability that $\|S_t\|_{V_t^{-1}}$ belongs to a convex set, when $S$ and $V$ obey a ``generalized canonical''  assumption (see \citep[Equation (14.5)]{pena2008self}). This assumption is similar in spirit to the sub-$\psi$ assumption of \citet{howard2020time} (extended to $\Re^d$) that we employ (Definition~\ref{def:sub-psi}). Unfortunately, outside of the sub-Gaussian case, explicit expressions for these sets appear unavailable.

In the contextual bandit setting with sub-Gaussian noise, \citet{abbasi2011improved} apply results of \citet{pena2004self} and give what is now a famous self-normalized inequality. They show that, if $S_t = \sum_{k\leq t} X_k \eps_k$ and $V_t = U_0 + \sum_{k\leq t}X_kX_k^\intercal $ where $\eps_k$ is scalar sub-Gaussian and $X_k\in\Re^d$ is predictable (i.e., measurable at time $k-1$) then $\|S_\tau\|_{V_\tau^{-1}} \lesssim \sqrt{\log(\det V_\tau)}$ at all stopping times $\tau$. This bound was extended to the infinite-dimensional case in \citet{abbasi2013online}.

Recently, \citet{whitehouse2023time} extended the general sub-$\psi$ condition of \citet{howard2020time} to $\Re^d$ and showed that, if $(S_t)$ and $(V_t)$ obey such a condition, then
\begin{equation}
\label{eq:whitehouse-1}
 \|S_\tau\|_{V_\tau^{-1}} \lesssim \gamma_{\min}^{1/2}(V_\tau)(\psi^*)^{-1}\left(\gamma_{\min}^{-1}(V_\tau)\left(\log\log(\gamma_{\max}(V_\tau)) + d\log(\kappa(V_\tau))\right)\right),
\end{equation}
where $\kappa(V_t) = \gamma_{\max}(V_t) / \gamma_{\min}(V_t)$ is the ratio of the largest to smallest eigenvalue of $V_t$ (the condition number) and $\psi^*$ is the convex conjugate of $\psi$. In the sub-Gaussian case we have $(\psi^*)^{-1}(x) \asymp \sqrt{x}$, so the bound of \citet{whitehouse2023time} scales as
\begin{equation}
\label{eq:whitehouse-2}
  \|S_\tau\|_{V_\tau^{-1}} \lesssim \sqrt{\log\log(\gamma_{\max}(V_\tau)) + d\log\kappa(V_\tau)}.
\end{equation}
This qualitatively new bound was also shown to be sharp with a corresponding lower bound.
Neither $\det V_\tau$ (the product of all eigenvalues) nor $\kappa(V_\tau)$  dominates the other, making the bound of \citet{whitehouse2023time} and \citet{abbasi2011improved} incomparable in general. If the matrices are well-conditioned then the former may be tighter, depending on the dimension. If there are many small eigenvalues (e.g., perhaps $V_t$ is the result of many rank-$k$ updates for $k\ll d$) then the determinant-based bound may be tighter. In any case, the gap between condition number-based bounds and determinant-based bounds is what motivates this work.

Throughout, we call a bound \emph{dimension-free} if it contains no separate factor depending only on the ambient dimension $d$. The dimension may nevertheless enter through matrix-dependent quantities such as $\log\det(V_t+U_0)$, which can itself scale proportionally to $d$. This is also the meaning of ``dimension-free'' in Table~\ref{tab:previous_work}.

Very recently, both \citet{akhavan2025bernstein} and \citet{metelli2025generalized} gave dimension-free, self-normalized Bernstein inequalities, building off the earlier work of \citet{faury2020improved} and \citet{zhou2021nearly}, both of which provide  dimension-dependent Bernstein-type bounds. \citet{akhavan2025bernstein} study the case in which the normalization process $(V_t)$ is the quadratic variation of $(S_t)$ and \citet{metelli2025generalized} study the bandit setting (as do \citet{faury2020improved} and \citet{zhou2021nearly}).  All of these results apply to bounded processes and do not extend to general sub-$\psi$ processes. They are closest to the Bennett and Bernstein results we give in Appendix~\ref{sec:bennett}, though we note that our Bernstein bound applies to random vectors satisfying the Bernstein condition---a more general condition than boundedness.

\begin{table}[t]
    \centering
    \small
    \renewcommand\arraystretch{1.16}
    \begin{tabular}{r| c | c| c }
     & \emph{Condition} & \emph{Dimension-free} & \emph{Dependence} \\
    \hline
    \citet{faury2020improved} & Bounded & & $\det V_t$ \\
    \citet{zhou2021nearly} & Bounded & & UB on $\|S_t\|$ \\
    \citet{ziemann2025vector} & Bounded & & $\det V_t$ \\
    \citet{akhavan2025bernstein}$^{**}$ & Bounded & \checkmark & $\det V_t$ \\
    \citet{metelli2025generalized}$^{**}$ & Bounded & \checkmark & $\det V_t$ \\
    \citet{kirschner2025confidence} & Bounded & & $\det V_t$ \\
    \citet{kirschner2025confidence} & Gaussian & \checkmark & $\det V_t$ \\
        \citet{pena2004self} &  sub-Gaussian & \checkmark & $\det V_t$ \\
        \citet{abbasi2011improved} & sub-Gaussian & \checkmark & $\det V_t$ \\
        \citet{abbasi2013online} & sub-Gaussian & \checkmark & $\det V_t$ \\
        \citet{pena2008self}$^*$ & Gen. canonical & \checkmark & $\det V_t$\\
        \citet{martinez2025vector}$^{**}$ & Bernstein & \checkmark & $\det V_t$ \\
        \citet{whitehouse2023time} & sub-$\psi$ & & $\kappa(V_t) $ \\
        {\bf This work } & sub-$\psi$ & \checkmark & $\det V_t$
    \end{tabular}
    \caption{A brief overview of previous vector-valued self-normalized concentration inequalities. ``Condition'' refers to the distributional assumptions made on the underlying process.
    ``UB on $\|S_t\|$'' means the bound is a function of the upper bound of the norm of the observations (i.e., it is not adaptive to $V_t$).
    \citet{kirschner2025confidence} provide two bounds: one under a Gaussian assumption and one under a bounded assumption.
    ``Gen. canonical'' refers to the generalized canonical assumption of \citet[Section 14.1.2]{pena2008self}. The difference between \citet{abbasi2011improved} and \citet{abbasi2013online} is that the latter holds in infinite-dimensional spaces.
    ($^*$) indicates that the results are not in closed-form.
    ($^{**}$) indicates that these works were concurrent to this one.}
    \label{tab:previous_work}
\end{table}

In the contextual bandit setting, \citet{martinez2025vector} also give a dimension-free determinant-based bound for random vectors satisfying the Bernstein condition. It differs from ours in several ways: It is time-uniform for the finite horizon $1\leq t\leq n$ for some fixed $n$, it has a multiplicative dependence between $\log(1/\delta)$ and $\log\det(V_t)$, and it requires a deterministic upper bound on $\log\det(V_t)$ to instantiate. Nevertheless, it has several appealing properties, and will dominate our bound in some regimes. We
provide a more involved discussion in  Section~\ref{sec:suboptimal}.  See Table~\ref{tab:previous_work} for an overview of known self-normalized concentration results in the multivariate setting.

Finally, let us discuss previous work on the variational approach to concentration, which undergirds our work here. This technique originated in the PAC-Bayesian approach to statistical learning theory as a method to bound the risk of randomized predictors.
As far as we are aware, the first to employ the variational approach was \citet{catoni2012challenging}, building off of new PAC-Bayes bounds developed in \citet{catoni2004statistical}. It was then deployed by both \citet{oliveira2016lower} and \citet{catoni2017dimension} in the multivariate setting, the former for bounding tails of quadratic forms, and the latter to bound the norm of random vectors.  Since then, a number of authors have employed it to solve various problems.

\citet{giulini2018robust} used the variational approach to estimate the Gram operator in Hilbert spaces. Around the same time,  \citet{catoni2018dimension} used it to bound the operator norm of random matrices in finite-dimensional Euclidean spaces. More recently, all of \citet{nakakita2024dimension,zhivotovskiy2024dimension,puchkin2024dimension,patarusau2025structured}
return to the question of matrix concentration/covariance estimation, each giving dimension-free bounds on sums of random matrices under various conditions.  \citet{chugg2025time} study the concentration of random vectors using the variational method, giving bounds under a variety of distributional assumptions (including sub-$\psi$ condition which we study here). All of the above works using the variational approach are in non-self-normalized settings.

\citet{kirschner2025confidence} and \citet{ziemann2025vector} recently noted that variational ideas can be used to obtain bounds similar to (or identical to) the ones presented by \citet{abbasi2011improved}. \citet{ziemann2025vector} also gives a dimension-dependent Bernstein inequality for bounded observations in a contextual bandit setting.
One can view our work here as extending the ideas of \citet{kirschner2025confidence} and \citet{ziemann2025vector} to much more general stochastic processes.

\subsection{Contributions and outline}
\label{sec:contributions}

In the sense defined above, our bounds are dimension-free: their dependence on the geometry of the process is through the log-determinant rather than a separate dimension-only factor.

The overarching goal of this work is to provide self-normalized bounds for sub-$\psi$ processes which depend on $\log \det V_\tau$ instead of $d\log \kappa(V_\tau)$.
After defining sub-$\psi$ processes and giving an overview of the variational approach in Section~\ref{sec:prelims}, Section~\ref{sec:sub-gaussian} provides a gentle introduction to our method, showing that we recover the bound of \citet{abbasi2011improved} exactly when the process is sub-Gaussian.

Section~\ref{sec:super-gaussian} then provides concentration results for general sub-$\psi$ processes evolving in $\Re^d$. In particular, Theorem~\ref{thm:sub-psi} provides a dimension-free, self-normalized ``line-crossing inequality,'' which gives the probability that $\|S_t\|_{(V_t+U_0)^{-1}}$ ever crosses a threshold which depends on a parameter $\lambda>0$. Different values of $\lambda$ optimize the threshold at different (intrinsic) times.

Section~\ref{sec:super-gaussian-stitching} then employs a method known as \emph{stitching}~\citep{howard2021time}, which iteratively applies Theorem~\ref{thm:sub-psi} with different parameters in different geometrically-spaced epochs of time, in order to give a bound which remains tight at all times. This procedure costs only an iterated logarithm term (this is similar to a doubling trick, chaining, peeling, etc.). The resulting bound, Theorem~\ref{thm:convex_conjugate_bound}, applies to arbitrary sub-$\psi$ processes and is stated in terms of the convex conjugate of $\psi$. Corollary~\ref{cor:sub-gamma-stitching} specializes it to sub-gamma processes, where the constants are small and explicit.

As an application, Section~\ref{sec:emp-bernstein} provides the first dimension-free, self-normalized empirical Bernstein inequality for bounded random vectors. It is `empirical' because it does not rely on a priori knowledge of the variance, but instead uses the empirical variance to adapt to the distribution over time.
Finally, in the spirit of more traditional applications, Appendix~\ref{sec:bennett} applies our results to obtain self-normalized (non-empirical) Bennett and Bernstein inequalities.

\section{Preliminaries}
\label{sec:prelims}
Fix a filtered probability space $(\Omega, \calF\equiv (\calF_t)_{t\geq 0}, P)$. We assume that we are working in discrete time, so $\calF$ is indexed by $\mathbb{N}\cup\{0\}$.
We consider two stochastic processes $(S_t)_{t\geq 0}$ and $(V_t)_{t\geq 0}$ on $(\Omega, \calF, P)$ which are adapted to $\calF$ (i.e., $S_t$ and $V_t$ are $\calF_t$-measurable). Throughout this work we assume that $S_t\in \Re^d$ and $V_t\in\Re^{d\times d}$ for some finite positive integer $d$, and that $S_0=0$. $S$ is a mnemonic for sum, and $V$ for variance.
We assume that $V_t$ is positive semi-definite for each $t\geq 0$, meaning that $\la \theta, V_t\theta\ra \geq 0$ for all $\theta\in\Re^d$. Our bounds always pair $V_t$ with a fixed positive-definite matrix $U_0$, so that $V_t+U_0$ is invertible even when $V_t$ is not. Typically, we will also have $V_t \preceq V_{t+1}$, where $\preceq$ is the positive semidefinite partial (Loewner) order.

We are interested in obtaining bounds on the regularized norm $\|S_\tau\|_{(V_\tau+U_0)^{-1}}$  for any $\calF$-adapted stopping time $\tau$ under various assumptions on the relationship between $S_t$ and $V_t$. More specifically, we think of $(V_t)$ as the accumulated variance process of $(S_t)$ or, in the words of \citet{blackwell1973amount}, a measure of the intrinsic time-scale of $(S_t)$. This is formalized with the notion of a sub-$\psi$ process, which generalizes many distributions of interest. Let $\dsphere := \{ x\in \Re^d: \|x\|_2=1\}$ denote the unit sphere in $\Re^d$. Recall that a process $(A_t)$ is a supermartingale with respect to $\calF$ if it is $\calF$-adapted and $\E[A_t|\calF_{t-1}] \leq A_{t-1}$ for all $t\geq 1$.

\begin{definition}[Sub-$\psi$ process in $\Re^d$]
\label{def:sub-psi}
    Let $\psi: [0,\lambda_{\max}) \to \Re_{\geq 0}$. Let $(S_t)$ be an $\Re^d$-valued process and $(V_t)$, $V_t\succeq0$ an $\Re^{d\times d}$-valued process, both adapted to a filtration $\calF$.
    We say that $(S_t)$ is a sub-$\psi$ process with variance proxy $(V_t)$ if, for all $\theta\in\dsphere$, all $\lambda\in[0,\lambda_{\max})$, and all $t\geq 0$,
    \begin{equation}
    \label{eq:sub-psi-def}
      \exp\big\{ \lambda \la \theta, S_t\ra - \psi(\lambda) \la \theta, V_t\theta\ra \big\} \leq L_t^\lambda(\theta),
    \end{equation}
    where $(L^\lambda_t(\theta))_{t\geq 0}$ is a nonnegative supermartingale adapted to $(\calF_t)$ with $L^\lambda_0(\theta)\leq1$. Equivalently, we say that $(S_t,V_t)$ is a sub-$\psi$ process.
\end{definition}

For scalar $S_t$ and $V_t$, sub-$\psi$ processes were originally proposed and studied by \citet{howard2020time} for scalar and matrix settings, generalizing and unifying a wide array of earlier work, such as \citet{pena2004self}. The vector case has been recently studied by \citet{whitehouse2023time}. Like both \citet{howard2020time} and \citet{whitehouse2023time}, we assume that $\psi$ is \emph{CGF-like}, meaning that it is strictly convex and twice continuously differentiable, with $\lim_{\lambda\to 0^+}\d \psi(\lambda)/\d\lambda = 0=\psi(0)$.

While the definition of a sub-$\psi$ process may appear abstract at first glance, let us reassure the reader that the definition captures many familiar (and unfamiliar) examples, including certain heavy-tailed settings.
\begin{enumerate}
    \item $\psi_N = \psi_N(\lambda) := \lambda^2/2$ for $\lambda\in[0,\lambda_{\max})$ results in a sub-exponential  process (sub-Gaussian if $\lambda_{\max} =\infty)$.
    If $X_t$ is $\Sigma_t$-sub-Gaussian\footnote{That is, $\E[\exp\{\lambda\la \theta, X_t\ra - \psi_N(\lambda)\la \theta, \Sigma_t\theta\ra\}|\calF_{t-1}] \leq 1$ for all $\theta\in\dsphere$ with $\lambda_{\max}=\infty$. In the scalar case, $X_t$ is $\sigma_t$-sub-Gaussian (conditional on the past)  if $\E[\exp(\lambda X_t)|\calF_{t-1}] \leq \exp(\psi_N(\lambda) \sigma_t^2)$ for all real $\lambda$.} then $S_t=\sum_{k\leq t}X_k$ is a sub-$\psi_N$ process with $\lambda_{\max}=\infty$ and $V_t = \sum_{k\leq t}\Sigma_k$. Here $(\Sigma_t)$ can be $\calF$-adapted (not just predictable).
    As a heavy-tailed example, if $(X_t)$ is a martingale difference sequence---that is, $\E[X_t|\calF_{t-1}]=0$---with $\E[\la \theta, X_t\ra^2|\calF_{t-1}]<\infty$ for all $\theta\in\dsphere$, then $S_t = \sum_{k\leq t}X_k$ is sub-$\psi_N$ with $\lambda_{\max}=\infty$ and $V_t = \frac{1}{3}\sum_{k\leq t} X_kX_k^\intercal + \frac{2}{3}\sum_{k\leq t}\E[X_kX_k^\intercal |\calF_{k-1}]$ (see \citet[Lemma 3f]{howard2020time}).
    More exotically,  consider any sequence of random vectors $(X_t)$ that are conditionally symmetric: $X_t \sim -X_t |\calF_{t-1}$ for all $t$.  Lemma 3 of \citet{pena2007pseudo} implies that $S_t = \sum_{k\leq t} X_k$ and $V_t = \sum_{k\leq t} X_kX_k^\intercal $ define a sub-$\psi_N$ process with $\lambda_{\max}=\infty$; note here that we did not require even a single moment to exist, let alone a moment generating function, in order to yield a sub-$\psi_N$ process. As in the scalar case, bounded random vectors with conditional mean zero can also be shown to be sub-$\psi_N$.
    \item $\psi_{G,c}(\lambda) =  \frac{\lambda^2}{2(1-c\lambda)}$ for $c\in\Re$ and $\lambda_{\max} = 1/\max\{c,0\}$ results in  a sub-gamma process.  The terminology stems from the fact that $\psi_{G,c}$ is an \emph{upper bound} on the CGF of a gamma random variable~\citep[Section 2.4]{boucheron2013concentration}. As in the scalar case, random vectors whose moments in each direction obey a Bernstein-type inequality can be shown to be sub-$\psi_{G,c}$; see Lemma~\ref{lem:bernstein-process}.  \citet{whitehouse2023time} further show that if $\E[X_t|\calF_{t-1}]=0$ and $\E|\la \theta,X_t\ra|^3$ is finite for all $\theta\in\dsphere$, then $S_t =\sum_{j\leq t}X_j$ is sub-$\psi_{G,c}$ for $c=1/6$ and $V_t = \sum_{j\leq t}(X_jX_j^\intercal + \E[\|X_j\|^3|\calF_{j-1}]I_d)$.
    Moreover, \citet[Proposition 1]{howard2020time} show that any twice differentiable, CGF-like $\psi$ can be bounded by $a\psi_{G,c}$ for some $a,c>0$. Thus, any sub-$\psi$ process is sub-$\psi_{G,c}$ after scaling (though transforming to sub-gamma may cause some looseness in the resulting bounds).
    \item $\psi_{E,c}(\lambda) = (-\log(1-c\lambda) - c\lambda)/c^2$ for $c\in\Re$ where $\lambda_{\max} =1/\max\{c,0\}$ results in a sub-exponential process, which is equivalent\footnote{In particular, Appendix E in \citet{howard2020time} shows that $\log \E e^{\lambda X} \leq \psi_{E,c}(\lambda)\sigma^2$ for all $0\leq \lambda<  1/c$ and some $\sigma,c>0$  if and only if $\log \E e^{\lambda X} \leq \psi_N(\lambda)u $ for all $0\leq \lambda < 1/b$ for some $u,b>0$.} up to constants to the sub-exponential condition mentioned under $\psi_N$. The terminology stems from the fact that $\psi_{E,c}$ is the CGF of a centered unit-rate negative-exponential random variable. Importantly, sums of centered, bounded random vectors can be shown to be sub-exponential with an \emph{adapted} (i.e.\ not $\calF$-predictable) variance $V_t$; see Section~\ref{sec:emp-bernstein}. Indeed, this observation undergirds our empirical Bernstein bound. As in the sub-gamma case, \citet[Proposition 1]{howard2020time} also shows that any CGF-like $\psi$ can be upper bounded by $a\psi_{E,c}$ for some $a,c\geq 0$ (but as mentioned above for sub-gamma, exploiting this transformation may yield looser bounds than directly dealing with $\psi$).
    \item $\psi_{P,c}(\lambda) = (e^{c\lambda} - c\lambda -1)/c^2$ for  $c\in\Re$ and $\lambda_{\max} = \infty$ results in a sub-Poisson process; this terminology stems from the fact that $\psi_{P,c}$ is the CGF of a centered unit-rate Poisson random variable. Most proofs of Bennett's inequality (e.g.,~\citet[Theorem 2.9]{boucheron2013concentration}) involve showing that a bounded random variable is sub-Poisson. In the multivariate setting, Lemma~\ref{lem:bennett-process} shows that conditionally centered vectors $(X_t)$ with $\|X_t\|\leq b$ are sub-$\psi_{P,b}$ with $V_t = \sum_{j\leq t} \E[X_jX_j^\intercal |\calF_{j-1}]$.
\end{enumerate}

\ifarxiv
For any $\psi$  such that $\psi(\lambda)/\lambda^2$ is nondecreasing in $\lambda$ (a property known as \emph{super-Gaussianity} which all of the above examples satisfy; see also Remark~\ref{rem:super-gaussian}), we can lift any scalar sub-$\psi$ process to a multivariate sub-$\psi$ process as follows. Consider $S_t=\sum_{k\leq t} \eta_k X_k$ where $(X_t)_{t\geq 1}$ is predictable, $\|X_t\|\leq 1$, and $(\eta_t)$ is a scalar-valued sub-$\psi$ process with variance proxy $(U_t)\subset\Re_{\geq 0}$.
Then $(S_t)$ is sub-$\psi$ with variance proxy $V_t = \sum_{k\leq t} U_k X_kX_k^\intercal $. This is the contextual bandit setting described in Section~\ref{sec:related-work}.
\fi

Let us now turn to the techniques we use to prove our results.
The bounds of \citet{pena2008self} and \citet{abbasi2011improved} are proved using the so-called ``method of mixtures'' (sometimes also called the pseudo-maximization technique~\citep{pena2007pseudo}). At its core, the method of mixtures involves (i) finding a family of supermartingales $Z(\theta)$ indexed by some parameter space $\Theta$, (ii) constructing a new supermartingale by integrating this family with respect to some measure $\rho$ over $\Theta$, say $Z_t = \int_\Theta Z_t(\theta) \d\rho$, and (iii) applying Markov's inequality (or, more precisely, Ville's inequality) to $Z_t$. For instance, one might let $Z_t(\theta)$ be the left hand side of~\eqref{eq:sub-psi-def} and $\rho$ be a measure over $\Theta=\dsphere$.

To prove our results, we instead turn to a related but distinct technique which we call the \emph{variational approach to concentration}.
Here we also begin with a mixture but after step (ii) apply the Donsker-Varadhan change of  measure formula in order to bound $Z_t$ in terms of the KL divergence between $\rho$ and some prior distribution. The benefit of this technique is that $\rho$ can be data-dependent.
This approach lies at the heart of (and originated in) PAC-Bayesian learning theory~\citep{guedj2019primer,alquier2024user}.  A generic template for such a bound, which applies to general stochastic processes and handles stopping times, is given by \citet[Theorem 4]{chugg2023unified}. \ifarxiv See also \citet{flynn2023improved} for a similar statement. \fi

\begin{proposition}[Variational Template]
\label{prop:variational_template}
Let $\Theta$ be a measurable parameter space. For each $\theta\in\Theta$, let $Z(\theta) \equiv (Z_t(\theta))_{t\geq 0}$ be a stochastic process upper bounded by a nonnegative supermartingale $L(\theta)\equiv (L_t(\theta))_{t\geq 0}$. Assume that all processes are adapted to the same filtration $\calF$ and that $L_0(\theta)\leq 1$ for each $\theta\in\Theta$.
Let $\nu$ be a data-free distribution over $\Theta$.
Then, for all $\delta\in(0,1)$, with probability $1-\delta$, for any $\calF$-adapted stopping time $\tau$ and any $\calF_\tau$-measurable distribution $\rho_\tau$ over $\Theta$,
\begin{equation}
    \int \log Z_\tau (\theta) \rho_\tau(\d\theta) \leq \kl(\rho_\tau\|\nu) + \log(1/\delta).
\end{equation}
\end{proposition}
Unlike approaches to multivariate concentration based on covering arguments or chaining, the variational approach can lead to dimension-free bounds if one is careful in their choice of $\rho_\tau$ and $\nu$.
In our case the KL-divergence will typically act as $\log \det (V_\tau)$, which is how we obtain our results.

\textbf{Notation.} We denote the eigenvalues of a matrix $M$ as $\gamma_{\min}(M) = \gamma_1(M)\leq \dots\leq \gamma_d(M) = \gamma_{\max}(M)$. We let $\|M\|_{\op} = \gamma_{\max}(M)$ be the operator norm of $M$.
\ifarxiv For a convex function $f:I\to\Re$, $I\subset\Re$, we denote its convex conjugate (Legendre–Fenchel transform) by $f^*(y) = \sup_{x\in I}(xy - f(x))$, which is also a convex function.
We let $\mathbb{N}_0 = \{0,1,2,\dots\}$ denote the  natural numbers including zero and $\Re_{>0}$ denote the positive real numbers. \fi We use $A\succeq B$ to signify that $A$ is greater than or equal to $B$ in the Loewner order, i.e., $A-B$ is positive semidefinite. For a vector $v$ and positive semidefinite matrix $A$, $\|v\|_A$ denotes the norm defined as  $\|v\|_A^2 = \la v,Av\ra$, where $\la \cdot,\cdot\ra$ is the usual dot product in $\Re^d$ unless otherwise specified. We let $a\vee b = \max\{a,b\}$ and  use $A^\intercal $ to refer to the transpose of the matrix $A$. \ifarxiv Throughout, the indices on the sum $\sum_{j\leq t}$ should be assumed to run from 1 to $t$. \fi

\section{Warmup: Sub-Gaussian Processes}
\label{sec:sub-gaussian}

Let us begin by considering the special case of sub-Gaussian processes\footnote{It's perhaps worth mentioning ``sub-Gaussian process'' has a double meaning in the literature. In this paper, a sub-Gaussian process refers to Definition~\ref{def:sub-psi} with $\psi = \psi_N$. This is different than the sub-Gaussian process studied by \citet{talagrand2014upper}, which is a collection of zero-mean random variables $\{X_\theta: \theta\in\Theta\}$ such that $\E[\exp(\lambda(X_\theta - X_\phi))]\leq \exp(\psi_N(\lambda)d^2(\theta, \phi))$ for some metric $d$ and all $\theta, \phi\in\Theta$.}. That is, we assume that $\psi(\lambda) = \psi_N(\lambda) = \lambda^2/2$. This allows for a particularly clean application of the variational template, since we may take $\Theta=\Re^d$ as our parameter space. (For $\psi=\psi_N$, one can replace $\dsphere$ with $\Re^d$.) This in turn enables the use of Gaussian distributions in Proposition~\ref{prop:variational_template}.

We emphasize that the following result is only a very modest generalization of known bounds. More specifically, it is the same bound of \citet{abbasi2011improved}, though in a slightly more general setting: they take the variance process $(V_t)$ to be predictable, whereas we require only that it be adapted. 
The result is of interest not for its (lack of) novelty, but because it demonstrates that the variational approach can recover the same bound with the same constants as the method of mixtures. \citet{ziemann2025vector} recently made a similar observation, though he focuses on the bandit setting in particular. What the method of mixtures does not appear to deliver is the general sub-$\psi$ bound of Section~\ref{sec:super-gaussian}, which is where the variational approach earns its keep. The proof can be found in Appendix~\ref{proof:sub-gaussian}.
\begin{theorem}
\label{thm:sub-gaussian}
    Let $(S_t,V_t)$ be a sub-Gaussian process. Let $U_0$ be a fixed positive-definite matrix. Then, for any $\delta\in(0,1)$, with probability $1-\delta$, for any stopping time $\tau$,
    \begin{equation}
    \label{eq:subGaussian_bound}
    \|S_\tau\|_{(V_\tau + U_0)^{-1}}^2 \leq \log\left(\frac{\det (V_\tau + U_0) }{\det  U_0 }\right) + 2\log(1/\delta).
\end{equation}
\end{theorem}
To give a sense of how Proposition~\ref{prop:variational_template} is used to prove Theorem~\ref{thm:sub-gaussian}, let us give a brief outline of the proof here.
If $(S_t,V_t)$ is a sub-Gaussian process, then it follows that  the process $(N_t(\theta))$ with $N_t(\theta) = \exp\{\lambda \la \theta, S_t\ra - \psi_N(\lambda) \la \theta, V_t\theta\ra\}$ is upper bounded by a supermartingale for all $\theta\in \Re^d$ (not only for every $\theta\in\dsphere$; this is because $\lambda_{\max}=\infty$). We may therefore apply Proposition~\ref{prop:variational_template} with $\Theta=\Re^d$. In particular, we take $\rho_\tau$ to be Gaussian with mean $(V_\tau + U_0)^{-1}S_\tau$ and covariance $(2 \psi_N(\lambda))^{-1}(V_\tau + U_0)^{-1}$, and $\nu$ to be Gaussian with mean 0 and covariance $(2\psi_N(\lambda))^{-1} U_0^{-1}$. (Recall that $\rho$ can be data-dependent while $\nu$ must be data-free.) Then $\int \log N_t(\theta)\d\rho_\tau = (\lambda - \psi_N(\lambda))\|S_\tau\|^2_{(V_\tau + U_0)^{-1}} + \text{(extra terms)}$ and $\kl(\rho_\tau\|\nu) = \frac{1}{2}\log(\det(V_\tau + U_0)/\det U_0) + \text{(the same extra terms)}$. Rearranging and taking $\lambda=1$ gives the result.

Arithmetic aside, the crux of the argument is noticing that we may work with $\Theta=\Re^d$ instead of $\Theta=\dsphere$, enabling us to choose $\rho$ and $\nu$ as Gaussians.  This strategy unfortunately won't be available to us for more general sub-$\psi$ distributions, which makes those settings more challenging. We will end up using uniform distributions instead.

Theorem~\ref{thm:sub-gaussian} recovers the bound of \citet{abbasi2011improved} if  $(X_t)_{t\geq 1}\subset\Re^d$ is an $\calF$-predictable process ($X_t$ is the $t$-th action taken by a bandit algorithm in their case) and $S_t = \sum_{k=1}^t \eta_k X_k$ where $\eta_k$ is scalar $\sigma$-sub-Gaussian noise.
Then $(S_t)$ is sub-$\psi_N$ with variance proxy $(V_t)$ defined as $V_t = \sigma^2\sum_{k=1}^t X_kX_k^\intercal $. Taking $U_0 = \sigma^2 V_0$ for some $V_0$ and rearranging~\eqref{eq:subGaussian_bound} gives Theorem 1 of \citet{abbasi2011improved}. More generally, Theorem~\ref{thm:sub-gaussian} applies to any sub-Gaussian process described in Section~\ref{sec:prelims}, such as $S_t = \sum_{k\leq t} X_k$ and $V_t = \sum_{k\leq t} X_kX_k^\intercal $ when $X_t$ is conditionally symmetric.

While Theorem~\ref{thm:sub-gaussian} is stated in terms of stopping times, it can equivalently be stated as a time-uniform bound of the form: $P(\forall t\geq 1: \eqref{eq:subGaussian_bound}\text{ holds})\geq 1-\delta$. The same equivalence holds for every bound we give. In this paper we state our bounds in terms of stopping times to conform to previous work in this area.

\section{General \texorpdfstring{Sub-$\psi$}{Sub-psi} Processes}
\label{sec:super-gaussian}

Theorem~\ref{thm:sub-gaussian} relied on transforming the sub-$\psi$ condition from one which holds over $\dsphere$ to one which holds over $\Re^d$, thus enabling us to take $\Theta=\Re^d$ and use normal distributions in Proposition~\ref{prop:variational_template}. This trick cannot be performed for general sub-$\psi$ processes because (in their definition) $\lambda_{\max}$ may be finite. However, for general $\psi$ functions we may replace $\dsphere$ with $\dball$, the unit ball in $\Re^d$. More specifically, we can upper bound any CGF-like $\psi$ with some super-Gaussian $\widehat{\psi}$ that allows us to replace $\dsphere$ with $\dball$;
see Remark~\ref{rem:super-gaussian} for further discussion. This allows us to use nested uniform distributions over $\Theta = \dball$ in Proposition~\ref{prop:variational_template}.

The machinery we use to handle general  sub-$\psi$ processes does have drawbacks over what was used in Section~\ref{sec:sub-gaussian} to study the specific case of sub-Gaussian processes. For one, when we apply the results in this section to sub-Gaussian processes, the resulting bound is looser than Theorem~\ref{thm:sub-gaussian}.
Second, the parameter $\lambda$ of the sub-Gaussian process in Theorem~\ref{thm:sub-gaussian} can be optimized independently of other terms in the bound. Such closed-form optimization of $\lambda$ is not possible in the more general case. We instead end up with \emph{line-crossing inequalities:} for any given $\lambda$, we obtain a bound on the probability that $\|S_t\|_{(V_t+U_0)^{-1}}$ ever crosses a threshold parameterized by $\lambda$. Different (predetermined) values of $\lambda$ result in a bound which is tighter at different (predetermined) values of $\det V_t$. Theorem~\ref{thm:sub-psi} gives the line-crossing inequality for general sub-$\psi$ processes.

Line-crossing inequalities are common in works studying the concentration of sub-$\psi$ processes~\citep{howard2020time,howard2021time}.
To go from line-crossing inequalities to bounds which are tight at all times, we deploy a technique which has come to be known as ``stitching'' (see Section~\ref{sec:super-gaussian-stitching} for more background).
We provide two kinds of stitched bounds: One for general sub-$\psi$ processes, and one specifically tailored for sub-gamma processes. This will be the goal of Section~\ref{sec:super-gaussian-stitching}.

\subsection{A line-crossing inequality for sub-$\psi$ processes}

We must define two quantities before stating our first result. First, for positive definite matrices $M_1$ and $M_2$, define their \emph{Rayleigh-Ritz quotient}\footnote{This is also the maximum of the generalized Rayleigh coefficient between $M_1$ and $M_2$, and the largest generalized eigenvalue of the problem $M_1x = \lambda M_2 x$.}
\begin{equation}
\label{eq:rayleigh}
    \alpha(M_1,M_2) \equiv \sup_{\theta\in\Re^d\setminus\{0\}} \frac{\la \theta, M_1\theta\ra}{\la \theta, M_2\theta\ra}.
\end{equation}
When using the Rayleigh-Ritz quotient hereafter, we will always consider matrices $M_1,M_2$ such that $M_2\succeq M_1$, thus implying that
\ifarxiv
$$0\leq \alpha(M_1,M_2) \leq 1.$$
\else
$0\leq \alpha(M_1,M_2) \leq 1.$
\fi
{Next, given a sub-$\psi$ process $(S_t,V_t)$ and a positive-definite
matrix $U_0$, let
\begin{equation}
\label{eq:alpha-t}
    \alpha_t := \alpha(U_0,U_0+V_t),
\end{equation}
and write $\upsilon=\gamma_{\min}(U_0)$. Define
\begin{equation}
\label{eq:gt}
    g_t := \sqrt{\upsilon}\left(1-\sqrt{\alpha_t}\right).
\end{equation}
If $V_t\succ 0$ then $\alpha_t<1$ and hence $g_t>0$. If instead $V_t$ is singular---as happens, for instance, when $V_t=\sum_{k\leq t}X_kX_k^\intercal $ has not yet attained full rank---then $\alpha_t=1$, $g_t=0$, and the bound below is vacuous. Meanwhile,
$g_t\to\sqrt{\upsilon}$ as $\alpha_t\to 0$, which occurs, for example,
if $\gamma_{\min}(V_t)\to\infty$. Finally, define the quantity
\begin{equation}
\label{eq:Delta-t}
    \Delta_t(\delta) := \frac{1}{2}\log\left(\frac{\det(V_t+U_0)}{\det U_0}\right) + \log(1/\delta),
\end{equation}
which will appear in many of the bounds below.}
We now present the line-crossing inequality for sub-$\psi$ processes.
The proof is in Appendix~\ref{proof:sub-psi}.

\begin{theorem}
\label{thm:sub-psi}
Let $(S_t,V_t)$ be a sub-$\psi$ process such that
$\lambda\mapsto\psi(\lambda)/\lambda^2$ is nondecreasing. Let $U_0$ be
positive definite, write $\upsilon=\gamma_{\min}(U_0)$, and define
$\alpha_t$ and $g_t$ as in~\eqref{eq:alpha-t} and~\eqref{eq:gt}.
Then, for any $\delta\in(0,1)$ and any
$\lambda\in(0,\lambda_{\max})$, with probability $1-\delta$, for any
stopping time $\tau$,
\begin{equation}
\label{eq:sub-psi-bound}
    \|S_\tau\|_{(V_\tau+U_0)^{-1}}
    \leq
    \frac{1}{\lambda g_\tau}
    \left( \Delta_\tau(\delta)
        + 2\upsilon\psi(\lambda)
    \right).
\end{equation}
\end{theorem}

\begin{remark}
\label{rem:super-gaussian}
    Theorem~\ref{thm:sub-psi} assumes that $\psi$ is super-Gaussian, i.e., that  $\psi(\lambda)/\lambda^2$ is  nondecreasing. This assumption is without loss of generality in the following sense: As we detailed in Section~\ref{sec:prelims}, any $\psi$ function can be upper bounded by a constant times $\psi_{G,c}$, which is super-Gaussian. (In fact, $\psi_N$ is super-Gaussian, as are $\psi_{G,c}$, $\psi_{E,c}$ and $\psi_{P,c}$ for every $c\geq 0$. For $c<0$ each of these three is strictly sub-Gaussian, and the reduction just described applies.) Thus, Theorem~\ref{thm:sub-psi} may be applied to any sub-$\psi$ process, at the cost of some looseness in constants if $\psi$ is strictly sub-Gaussian (e.g., sub-Bernoulli).
\end{remark}

\begin{remark}
\label{rem:sharper-constant}
    The proof of Theorem~\ref{thm:sub-psi} in fact delivers the slightly sharper bound
    \begin{equation}
    \label{eq:sub-psi-bound-sharp}
        \|S_\tau\|_{(V_\tau+U_0)^{-1}}
        \leq
        \frac{1}{\lambda g_\tau}
        \left(
        \Delta_\tau(\delta)
            + \left(1+\left(1-\sqrt{\alpha_\tau}\right)^2\right)\upsilon\psi(\lambda)
        \right),
    \end{equation}
    in which the constant 2 multiplying $\upsilon\psi(\lambda)$ is replaced by $1+(1-\sqrt{\alpha_\tau})^2$. The two coincide as $\alpha_\tau\to0$, but~\eqref{eq:sub-psi-bound-sharp} is up to twice as tight in that term when $\alpha_\tau$ is close to 1, i.e., in the early rounds before $\gamma_{\min}(V_\tau)$ has grown. We state~\eqref{eq:sub-psi-bound} in the theorem because the constant is cleaner and the difference is immaterial asymptotically, but we use~\eqref{eq:sub-psi-bound-sharp} when plotting our bounds.
\end{remark}

We sketch the proof, emphasizing what differs from that of Theorem~\ref{thm:sub-gaussian}; the full argument is in Appendix~\ref{proof:sub-psi}. The skeleton is the same: we apply Proposition~\ref{prop:variational_template} to $Z_t(\theta) = \exp\{\lambda\la\theta,S_t\ra - \psi(\lambda)\la\theta,(V_t+U_0)\theta\ra\}$ with a data-dependent $\rho$ and a data-free $\nu$, and the KL divergence supplies the determinant term. Two things change. First, since $\lambda_{\max}$ may be finite, we can no longer extend the sub-$\psi$ condition from $\dsphere$ to all of $\Re^d$; super-Gaussianity lets us extend it to the unit \emph{ball} $\dball$ instead, so we take $\rho$ and $\nu$ uniform over ellipsoids $\calE_\rho,\calE_\nu\subset\dball$ rather than Gaussians over $\Re^d$. Their common scale $\upsilon=\gamma_{\min}(U_0)$ keeps $\calE_\nu$ inside $\dball$ and cancels in $\kl(\rho\|\nu) = \tfrac{1}{2}\log(\det(V_\tau+U_0)/\det U_0)$. Second, the mean of $\calE_\rho$ can be shifted only so far as to keep it inside $\calE_\nu$: the largest admissible shift equals $g_\tau$, and carrying the resulting $g_\tau^2\psi(\lambda)$ across the inequality---together with the covariance contribution of $\rho$, each bounded by $\upsilon\psi(\lambda)$---yields the $2\upsilon\psi(\lambda)$ term of~\eqref{eq:sub-psi-bound}.

\begin{remark}
\label{rem:lambda-increasing}
    In many non-self-normalized concentration inequalities (e.g., traditional Hoeffding, Bernstein, Bennett bounds) one typically chooses $\lambda$ as a decreasing function of $n$. That is not the case with the self-normalized inequalities studied here, whose widths grow with time. For this reason the value of $\lambda$ minimizing~\eqref{eq:sub-psi-bound} at a particular time $t=n$ is an \emph{increasing} function of $n$. In the sub-Gaussian case a short calculation gives $\lambda_* = \sqrt{\Delta_n(\delta)/\upsilon}$, which for a fixed $U_0$ grows like $\sqrt{\log\det (V_n) + \log(1/\delta)}$.
\end{remark}

It's worth asking how Theorem~\ref{thm:sub-psi} compares with Theorem~\ref{thm:sub-gaussian}, even though the former handles far more general distributions.
At first glance, Theorem~\ref{thm:sub-psi} may appear as though it has the wrong order of magnitude. Indeed, Theorem~\ref{thm:sub-gaussian} is a bound on $\|S_\tau\|_{(V_\tau+U_0)^{-1}}^2$, whereas~\eqref{eq:sub-psi-bound} is a bound on $\|S_\tau\|_{(V_\tau+U_0)^{-1}}$ (without the square). To set ourselves at ease, minimize the right hand side of~\eqref{eq:sub-psi-bound} over $\lambda$. Since $\inf_\lambda (x+\psi(\lambda))/\lambda = (\psi^*)^{-1}(x)$ for the convex conjugate $\psi^*$ of $\psi$, this gives
\begin{equation}
\label{eq:sub-psi-optimized}
    \inf_{\lambda\in(0,\lambda_{\max})} \frac{\Delta_\tau(\delta) + 2\upsilon\psi(\lambda)}{\lambda g_\tau} = \frac{2\sqrt{\upsilon}}{1-\sqrt{\alpha_\tau}}(\psi^*)^{-1}\left(\frac{\Delta_\tau(\delta)}{2\upsilon}\right).
\end{equation}
For $\psi = \psi_N$ we have $(\psi_N^*)^{-1}(x) = \sqrt{2x}$, so the right hand side is $2\sqrt{\Delta_\tau(\delta)}/(1-\sqrt{\alpha_\tau})$ and thus
\[\|S_\tau\|^2_{(V_\tau+U_0)^{-1}} \leq \frac{4\Delta_\tau(\delta)}{(1-\sqrt{\alpha_\tau})^2}.\]
This is of the same order as Theorem~\ref{thm:sub-gaussian}, which reads $\|S_\tau\|^2_{(V_\tau+U_0)^{-1}}\leq 2\Delta_\tau(\delta)$, and matches it up to a factor of $2/(1-\sqrt{\alpha_\tau})^2$, which tends to 2 as $\alpha_\tau\to 0$. However, all of this demonstrates only that there \emph{exists} a $\lambda$ which makes the widths comparable. We cannot in fact choose the minimizing $\lambda$ since it cannot be data-dependent. Obtaining a bound which behaves like~\eqref{eq:sub-psi-optimized} at all times is the purpose of the following section.

Figure~\ref{fig:subgamma} compares Theorem~\ref{thm:sub-psi} to the bound presented by \citet{whitehouse2023time}. We fix the growth rate of the condition number, $\kappa(V_t)$, and vary that of $\det(V_t)$. As expected, for all of the plotted values of $\lambda$, our bound outperforms the condition number-based bound of \citet{whitehouse2023time} when the determinant grows slowly but is uniformly weaker when the determinant grows quickly.

Intriguingly, the bound of \citet{whitehouse2023time} seems to always dominate our bound at small sample sizes (roughly $<200$). This is due to the factor of $g_t$ in our bounds, which is small at small $t$ (or, more specifically, until $\gamma_{\min}(V_t+U_0)$ grows sufficiently large relative to $\gamma_{\max}(U_0)$). This highlights a drawback of our bounds: if $\gamma_{\min}(V_t+U_0)$ never grows, or grows extremely slowly, then $\alpha_t$ remains close to 1, $g_t$ remains close to 0, and the bound blows up.

\begin{figure}[t]
  \centering
  \includegraphics[width=1\linewidth]{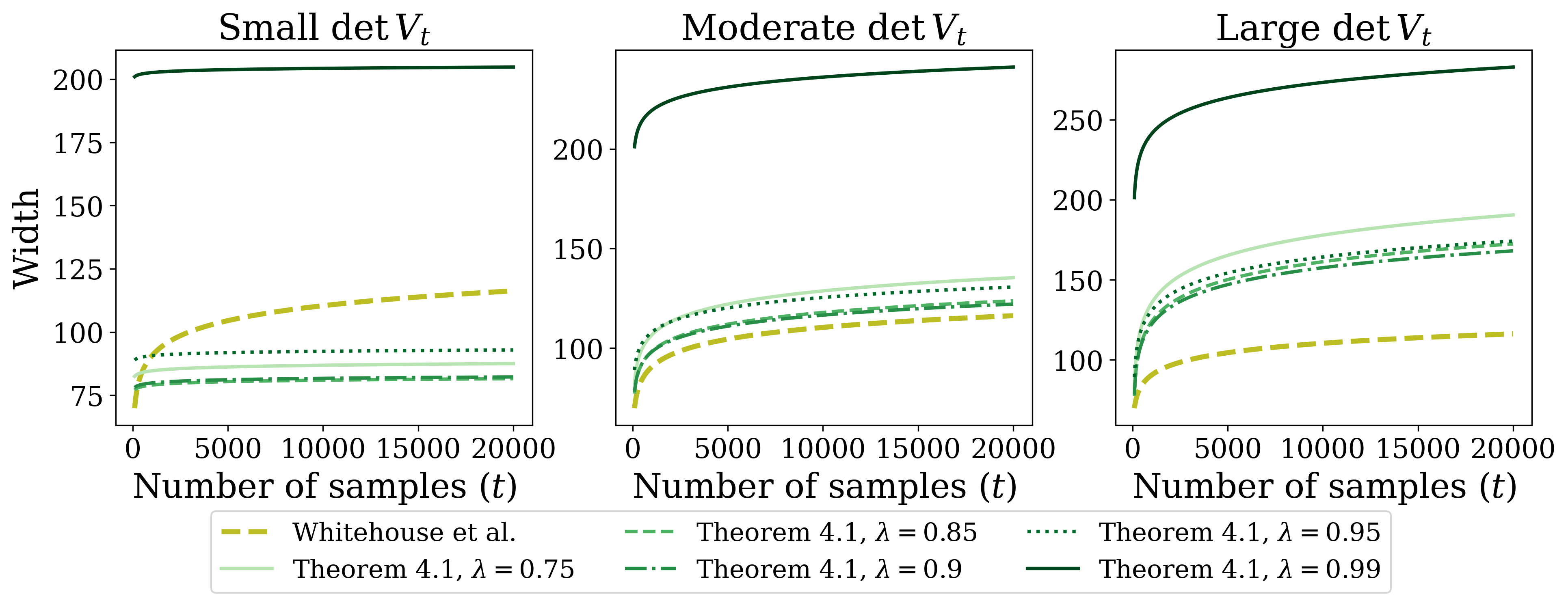}
  \caption{A comparison of Theorem~\ref{thm:sub-psi} with the bound of \citet{whitehouse2023time}, which we recall is based on the condition number of $V_\tau$. We control the growth of the determinant with rank-$k$ updates: as $k$ grows, so does $\det V_\tau$. The condition number has the same growth in each case.  As the sample size grows, Theorem~\ref{thm:sub-psi} outperforms the condition number-based bound for smaller values of $\det V_\tau$. It loses ground as the determinant grows relative to $d \log \kappa(V_\tau)$. We use $d = 20, c=1$ and $U_0 = I_d$.
  Full simulation details can be found in Appendix~\ref{app:experimental-details}.  }
  \label{fig:subgamma}
\end{figure}

\subsection{Optimizing $\lambda$ via Stitching}
\label{sec:super-gaussian-stitching}

In order to overcome the hurdle of optimizing $\lambda$ in Theorem~\ref{thm:sub-psi}, we'll apply the bound of Theorem~\ref{thm:sub-psi} iteratively over geometrically spaced epochs, choosing a different $\lambda$ in each epoch. This broad technique goes by many names such as doubling, discrete mixtures, chaining, etc., but here we follow a particularly sharp variant called stitching by \citet{howard2021time}.

The epochs are defined relative to the intrinsic time $\det (V_t+U_0)$. More specifically, the $k$-th epoch is $E_k = \{t: \eta^k \leq \det (V_t+U_0) < \eta^{k+1}\}$ for some parameter $\eta>1$.
The weight assigned to the $k$-th epoch is determined by a ``stitching function'' $\ell(k)$, which obeys $\sum_k \ell^{-1}(k)\leq 1$. We make the simplifying choices $\eta = 2$ and $\ell(k) = (k+1)^2\zeta(2)$, where $\zeta$ is the Riemann zeta function, so that $\zeta(2) = \pi^2/6$.

The resulting bound is most naturally expressed in terms of the convex conjugate of $\psi$,
\begin{equation}
\label{eq:convex-conjugate}
    \psi^*(u) = \sup_{\lambda\in[0,\lambda_{\max})}\left\{\lambda u - \psi(\lambda)\right\},
\end{equation}
for which we write $(\psi^*)^{-1}$ for the generalized inverse
\begin{equation}
\label{eq:convex-conjugate-inverse}
    (\psi^*)^{-1}(x) = \inf\left\{u\geq 0: \psi^*(u)\geq x\right\},
\end{equation}
with the convention that the infimum of the empty set is $+\infty$. For the standard $\psi$ functions considered in this paper, $\psi^*$ is strictly increasing and unbounded, so this is simply the ordinary inverse. Before stating the result, it is worth recording the rate it delivers. Ignoring constants, and provided $\alpha_\tau$ is bounded away from 1, Theorem~\ref{thm:convex_conjugate_bound} says that
\begin{equation}
\label{eq:convex-conjugate-rate}
    \|S_\tau\|_{(V_\tau+U_0)^{-1}}
    =
    O\left(
        \sqrt{\upsilon}\,(\psi^*)^{-1}\left(
            \frac{\log\det(V_\tau+U_0) + \log(1/\delta)}{\upsilon}
        \right)
    \right),
\end{equation}
where $\upsilon=\gamma_{\min}(U_0)$. (The iterated logarithm contributed by the stitching argument is of lower order and has been absorbed.) This is the same functional form as the bound~\eqref{eq:whitehouse-1} of \citet{whitehouse2023time}, with the dimension-dependent $d\log\kappa(V_\tau)$ replaced by $\log\det(V_\tau+U_0)$, but with $\upsilon$ in place of their $\gamma_{\min}(V_\tau)$, a difference we return to in Section~\ref{sec:suboptimal}. For $\psi=\psi_N$ we have $(\psi_N^*)^{-1}(x)=\sqrt{2x}$ and the right hand side of~\eqref{eq:convex-conjugate-rate} is $O(\sqrt{\log\det(V_\tau+U_0)+\log(1/\delta)})$, recovering the rate of \citet{abbasi2011improved}. The proof of the following result is given in Appendix~\ref{proof:convex_conjugate}.

\begin{theorem}
\label{thm:convex_conjugate_bound}
    Let $(S_t,V_t)$ be a sub-$\psi$ process, where $\psi$ is CGF-like
    and $\lambda\mapsto\psi(\lambda)/\lambda^2$ is nondecreasing. Let
    $U_0$ be positive definite and suppose that
    $V_t+U_0\succeq I_d$ for all $t\geq1$. Set $\upsilon=\gamma_{\min}(U_0)$ and $\alpha_t=\alpha(U_0,V_t+U_0)$.
    Then, for every $\delta\in(0,1)$, with probability at least
    $1-\delta$, for every stopping time $\tau\geq1$,
    \begin{equation}
    \label{eq:convex-conjugate-bound}
        \|S_\tau\|_{(V_\tau+U_0)^{-1}}
        \leq
        \frac{2\sqrt{\upsilon}}
             {1-\sqrt{\alpha_\tau}}
        (\psi^*)^{-1}\left(
            \frac{D_\tau}{2\upsilon}
        \right),
    \end{equation}
    where
    \begin{equation}
    \label{eq:Dtau}
        D_\tau
        =
        \frac{1}{2}
        \log\left(
            \frac{\det(V_\tau+U_0)}{\det U_0}
        \right)
        +
        1
        +
        2\log\left(
            \log_2\big(\det(V_\tau+U_0)\big)+1
        \right)
        +
        \log(1/\delta).
    \end{equation}
\end{theorem}
Sub-gamma processes give a bound with small and explicit constants. Indeed, for $\psi_{G,c}(\lambda) = \lambda^2/(2(1-c\lambda))$ a direct calculation gives
\begin{equation}
\label{eq:sub-gamma-conjugate-inverse}
    (\psi_{G,c}^*)^{-1}(x) = \sqrt{2x}+cx,
\end{equation}
and substituting this into~\eqref{eq:convex-conjugate-bound} gives the following.

\begin{corollary}
\label{cor:sub-gamma-stitching}
    Let $(S_t,V_t)$ be a sub-$\psi_{G,c}$ process for some $c>0$, and let $U_0$ be positive definite with $V_t+U_0\succeq I_d$ for all $t\geq 1$. Then, for every $\delta\in(0,1)$, with probability at least $1-\delta$, for every stopping time $\tau\geq1$,
    \begin{equation}
    \label{eq:sub-gamma-bound}
        \|S_\tau\|_{(V_\tau+U_0)^{-1}}
        \leq
        \frac{1}{1-\sqrt{\alpha_\tau}}
        \left(
            \frac{cD_\tau}{\sqrt{\upsilon}}
            +
            2\sqrt{D_\tau}
        \right),
    \end{equation}
    for $\upsilon$, $\alpha_\tau$ and $D_\tau$ as in Theorem~\ref{thm:convex_conjugate_bound}.
\end{corollary}
As is typical for stitched inequalities, the resulting boundary may be looser for moderate sample sizes than Theorem~\ref{thm:sub-psi} with a well-chosen fixed value of $\lambda$. Its advantage is that it removes $\lambda$ as a tuning parameter while retaining the appropriate inverse-conjugate rate uniformly over time.
The factor $(1-\sqrt{\alpha_\tau})^{-1}$ in Theorem~\ref{thm:convex_conjugate_bound} should be viewed as a penalty we pay for the appearance of $g_\tau$ in Theorem~\ref{thm:sub-psi}. It is finite as soon as $V_\tau\succ0$, and tends to 1 as $\gamma_{\min}(V_\tau)\to\infty$. To get a feel for its size, note that
\ifarxiv
\[\alpha_\tau \leq \frac{\gamma_{\max}(U_0)}{\gamma_{\min}(V_\tau + U_0)},\]
\else
$\alpha_\tau \leq \gamma_{\max}(U_0)/\gamma_{\min}(V_\tau + U_0),$
\fi
so the factor is at most 2 once $\gamma_{\min}(V_\tau + U_0)\geq 4\gamma_{\max}(U_0)$.
We suspect that the appearance of $(1-\sqrt{\alpha_\tau})^{-1}$ in~\eqref{eq:sub-gamma-bound} is sub-optimal, but it's unclear how to remove this term with our current proof techniques.

{Theorem~\ref{thm:convex_conjugate_bound} assumes that $V_t+U_0\succeq I_d$ to ensure that the union of the epochs $E_k = \{t: \eta^k \leq \det( V_t+U_0)<\eta^{k+1}\}$, $k\geq 0$ used in the stitching argument constitutes a partitioning of $\{t\geq1\}$. This costs no generality, since we may always rescale. Indeed, as $U_0$ is positive definite we have $V_t+U_0\succeq \upsilon I_d$ for all $t$, so if the assumption fails we may take $\beta=\upsilon<1$. If $(S_t,V_t)$ is a sub-$\psi$ process then \citet[Proposition 2.4]{whitehouse2023time} shows that $(S_t/\sqrt{\beta}, V_t/\beta)$ is a sub-$\psi_\beta$ process for $\psi_\beta(\lambda) = \beta \psi(\lambda/\sqrt{\beta})$, and $\psi_\beta$ is again CGF-like and super-Gaussian. We apply Theorem~\ref{thm:convex_conjugate_bound} to this process with the matrix $U_0/\beta$ in place of $U_0$, which leaves the quantity being bounded unchanged:
\ifarxiv
\[\left\|\frac{S_\tau}{\sqrt{\beta}}\right\|_{(V_\tau/\beta + U_0/\beta)^{-1}} = \|S_\tau\|_{(V_\tau+U_0)^{-1}}.\]
\else
$\|S_\tau/\sqrt{\beta}\|_{(V_\tau/\beta + U_0/\beta)^{-1}} = \|S_\tau\|_{(V_\tau+U_0)^{-1}}.$
\fi
The hypothesis of the theorem becomes $V_t+U_0\succeq \beta I_d$, which holds by construction. To read off the resulting boundary, note that the rescaled regularizer has $\gamma_{\min}(U_0/\beta)=1$, that $\alpha_\tau$ is unchanged (it is invariant to a common rescaling of both arguments), and that the conjugate of $\psi_\beta$ satisfies
\ifarxiv
\[\psi_\beta^*(u) = \beta\psi^*(u/\sqrt{\beta}), \qquad (\psi_\beta^*)^{-1}(x) = \sqrt{\beta}(\psi^*)^{-1}(x/\beta).\]
\else
$\psi_\beta^*(u) = \beta\psi^*(u/\sqrt{\beta})$ and $(\psi_\beta^*)^{-1}(x) = \sqrt{\beta}(\psi^*)^{-1}(x/\beta)$.
\fi
Rescaling leaves the ratio $\det(V_\tau+U_0)/\det U_0$ in $D_\tau$ untouched, but the determinant inside the iterated logarithm becomes $\beta^{-d}\det(V_\tau+U_0)$. Writing $D_\tau^\beta$ for the resulting quantity, the bound reads $\|S_\tau\|_{(V_\tau+U_0)^{-1}} \leq 2\sqrt{\beta}(1-\sqrt{\alpha_\tau})^{-1}(\psi^*)^{-1}(D^\beta_\tau/(2\beta))$, which is~\eqref{eq:convex-conjugate-bound} with $\upsilon$ replaced by $\beta$.}

\begin{remark}
The appearance of both the logarithmic term $\log\det V_\tau$ and the iterated logarithmic term $\log\log\det V_\tau$ in Corollary~\ref{cor:sub-gamma-stitching} may seem strange. They have different origins: the latter is a result of stitching while the former comes from the variational proof technique and, in particular, is the KL-divergence between the prior and posterior in Proposition~\ref{prop:variational_template}. In short, the term $\log\det V_\tau$ captures the complexity of the process. It replaces $d\log\kappa(V_\tau)$ in the bounds of \citet{whitehouse2023time} and is comparable to the dependence on variance-like terms in non-self-normalized bounds (see, e.g., \citealt[Theorem 2.3]{chugg2023unified}).
\end{remark}

\begin{remark}
    Related to the previous remark, it's worth noting that the appearance of $\log\det V_\tau$ has unfortunate consequences when the results are applied in $d=1$. In this case, $\log \kappa(V_\tau)=0$ and the bounds of \citet{whitehouse2023time} scale as $\|S_\tau\|_2 \lesssim \sqrt{V_\tau}(\psi^*)^{-1}(\log\log(V_\tau) + \log(1/\delta))$. In our case, however, we retain a $\log(V_\tau)$ term, which is unusual for time-uniform bounds in the scalar setting (cf. \citealp{howard2021time}). In particular, one might expect the term $\log\log(V_\tau)$ instead of $\log(V_\tau)$. In the multivariate setting, however, an iterated logarithm dependence on $\det V_\tau$ only is impossible, as shown by the following example.
\end{remark}

\begin{example}
\label{ex:lower_bound}
    We borrow the conclusion of an elegant example from \citet{pena2007pseudo}, also studied by \citet{whitehouse2023time} (see their Example 4.6). In particular, \citet{pena2007pseudo} construct a sub-Gaussian process $(S_t,V_t)$ in $\Re^2$ such that $\|S_t\|_{V_t^{-1}}\sim\sqrt{\log(t)}$, $\gamma_{\max}(V_t) \sim t(1 + a)$, and $\gamma_{\min}(V_t)\sim \log(t) / (1 + a)$ for a finite random variable $a$.   Hence $\det V_t \sim t\log t$, implying that $\|S_t\|_{V_t^{-1}}\sim \sqrt{\log(\det V_t)}$. This matches Theorem~\ref{thm:sub-gaussian} to leading order, and is incompatible with a rate of $\sqrt{\log\log(\det V_t)}$.
\end{example}

We also briefly note that Example~\ref{ex:lower_bound} shows that, even in the fixed-time setting, one cannot hope for bounds that do not have logarithmic dependence on $V_t$.
This is in contrast to the non-self-normalized fixed-time setting, where such dependence can be avoided.

\subsection{A note on suspected suboptimality}
\label{sec:suboptimal}
Consider Corollary~\ref{cor:sub-gamma-stitching} with
$V_t$ such that $\gamma_{\min}(V_t)\to\infty$, so that $\alpha_t\to0$, and $c=1$.
With probability $1-\delta$, we obtain that
\ifarxiv
\[\|S_t\|_{(V_t + U_0)^{-1}} \lesssim \frac{\log\det(V_t) + \log(1/\delta)}{\sqrt{\upsilon}} + \sqrt{\log\det(V_t) + \log(1/\delta)}.\]
\else
$\|S_t\|_{(V_t + U_0)^{-1}} \lesssim \frac{\log\det(V_t) + \log(1/\delta)}{\sqrt{\upsilon}} + \sqrt{\log\det(V_t) + \log(1/\delta)}.$
\fi
To put this in more familiar terms, suppose that $V_t = tI_d$. Then we obtain the bound
\[\left\|\frac{S_t}{t}\right\| \lesssim \frac{d\log(t) + \log(1/\delta)}{\sqrt{\upsilon t}} + \sqrt{\frac{d\log(t) + \log(1/\delta)}{t}}.\]
This resembles the terms in a Bernstein bound, except for a missing factor of $1/\sqrt{t}$ on the first term. Indeed, we expect the linear term in a Bernstein bound to shrink at rate $1/t$, not $1/\sqrt{t}$, thus making it asymptotically smaller than the square root term. This makes us suspect that Corollary~\ref{cor:sub-gamma-stitching} is suboptimal in terms of dependence on $D_t$. We imagine that the correct dependence would resemble $cD_t / \sqrt{\gamma_{\min}(V_t)} + \sqrt{D_t}$, akin to the bound of \citet{whitehouse2023time} (though, of course, theirs is not a log-determinant bound and not dimension-free).

It is instructive to compare our bound with several
concurrent works which study log-determinant Bernstein bounds~\citep{metelli2025generalized,akhavan2025bernstein,martinez2025vector}. Like us, \citet{akhavan2025bernstein} obtain a bound which scales linearly in $\log\det(V_t)$. Meanwhile, both \citet{metelli2025generalized} and \citet{martinez2025vector} avoid this linear dependence, but at the cost of a multiplicative dependence between $\log\det(V_t)$ and $\log(1/\delta)$. That is, their bounds scale as
\begin{equation*}
    \left\|S_t\right\|_{(V_t + U_0)^{-1}} \lesssim \log(1/\delta) + \sqrt{\log(\det(V_t))\log(1/\delta)}.
\end{equation*}

Our bound will thus be tighter for $\log(\det(V_t)) \lesssim \log(1/\delta)$, and theirs when $\log(1/\delta) \lesssim \log(\det(V_t))$.
There are other differences as well. For instance, the result of \citet{metelli2025generalized} only applies to bounded observations, and  those of \citet{martinez2025vector} are either not time-uniform or require a priori upper bounds on $\log\det(V_t)$.
Whether one can obtain a time-uniform log-determinant bound for sub-gamma processes with (i) additive dependence between $\log\det(V_t)$ and $\log(1/\delta)$ and (ii) a linear term vanishing at  a faster rate, is thus an open question.

\section{An Empirical Bernstein Inequality}
\label{sec:emp-bernstein}

In this section we provide a self-normalized ``empirical Bernstein'' result. Empirical Bernstein bounds are useful because they do not require a priori knowledge of the variance in order to be instantiated. Instead, they adapt to the unknown variance. Such bounds have been studied in scalar random variables~\citep{maurer2009empirical,waudby2024estimating,howard2021time,chugg2025closed}, vector-valued random variables~\citep{kohler2017sub,martinez2024empirical,martinez2025sharp,chugg2025time}, U-statistics~\citep{peel2010empirical}, and matrices~\citep{wang2024sharp,chugg2025closed}. As far as we are aware, \citet{whitehouse2023time} gave the first empirical Bernstein result for self-normalized, vector-valued processes. Ours is thus the first dimension-free, self-normalized, empirical Bernstein inequality.

Let $(X_t)$ be a stream of random vectors with conditional mean $\mu$. Suppose that $\|X_t\|\leq 1/2$ (one can replace 1/2 with any constant and rescale the bound accordingly). Recall that $\psi_{E,1}(\lambda) = -\log(1-\lambda) - \lambda$. For any adapted sequence $(\muhat_t)_{t\geq 0}$,\footnote{That is, $\muhat_t$ is $\calF_t$-measurable.} the process defined by
\begin{equation}
\label{eq:emp-bernstein-nsm}
    N_t(\theta) = \prod_{k\leq t}\exp\left\{\lambda \la \theta, X_k-\mu\ra - \psi_{E,1}(\lambda)\la \theta, X_k - \widehat{\mu}_{k-1}\ra^2\right\},\quad N_0(\theta) = 1.
\end{equation}
is a nonnegative supermartingale (see \citet[Lemma E.1]{chugg2025time} for an explicit proof).
It is a multivariate extension of a supermartingale used by \citet{howard2021time}, which they use to prove an empirical Bernstein bound for scalar random variables. This supermartingale is, in turn, an extension of an inequality of~\citet{fan2015exponential}, where the mean is replaced by the key term $\widehat{\mu}_{k-1}$ in the empirical variance. This reduces the variance when the mean of $X_k$ is nonzero and delivers a certain asymptotic sharpness property of the resulting confidence set widths in scalar, vector, and matrix settings~\citep{waudby2024estimating,martinez2024empirical,wang2024sharp}.

That the process in~\eqref{eq:emp-bernstein-nsm} is a supermartingale implies that $(S_t,V_t)$ is a sub-$\psi_{E,1}$ process where
\begin{equation}
\label{eq:eb-process}
    S_t = \sum_{k\leq t}(X_k-\mu), \quad V_t = \Gamma + W_t, \quad W_t =  \sum_{k\leq t} (X_k - \widehat{\mu}_{k-1})(X_k - \widehat{\mu}_{k-1})^\intercal ,
\end{equation}
and $\Gamma$ is any fixed PSD matrix (could be the zero matrix). We give the empirical variance process $W_t$ its own name because the two will need to be distinguished below.
Since $\psi_{E,1}$ is super-Gaussian, we may apply Theorem~\ref{thm:sub-psi} to obtain the following result.

\begin{corollary}
\label{cor:empirical-bernstein}
Let $(X_t)$ be an $\calF$-adapted stream of random vectors with conditional mean $\mu$, that is, $\E[X_t|\calF_{t-1}]=\mu$ for all $t\geq1$, and suppose $\|X_t\|_2\leq 1/2$ for all $t\geq 1$. Let $(\muhat_t)_{t\geq0}$ be any $\calF$-adapted sequence, let $\Gamma$ be a fixed PSD matrix, and let $V_t$ be as in~\eqref{eq:eb-process}. Let $U_0$ be positive definite, write $\upsilon = \gamma_{\min}(U_0)$, and let $\alpha_t$ and $g_t$ be as in~\eqref{eq:alpha-t} and~\eqref{eq:gt}.
    Then, for any $\delta\in(0,1)$ and any $\lambda\in(0,1)$, with probability $1-\delta$, for any stopping time $\tau$,
    \begin{equation*}
         \bigg\|\sum_{k\leq \tau} (X_k - \mu)\bigg\|_{(V_\tau+U_0)^{-1}} \leq \frac{\Delta_\tau(\delta) + 2\upsilon\psi_{E,1}(\lambda)}{\lambda g_\tau},
    \end{equation*}
    for $\Delta_\tau(\delta)$ as in~\eqref{eq:Delta-t}.
\end{corollary}
The upper bound of 1 on $\lambda$ comes from the fact that $\lambda_{\max}=1$ for $\psi_{E,1}$. We added the PSD matrix $\Gamma$ to the process $(V_t)$ in Corollary~\ref{cor:empirical-bernstein} because we can often obtain better bounds in practice by splitting $U_0 = \Gamma + U_0'$. To elaborate, suppose we want to bound $\|S_t\|_{(W_t + U_0)^{-1}}$ for some fixed $U_0$. Applying Corollary~\ref{cor:empirical-bernstein} directly means taking $\Gamma=0$, so that $V_t=W_t$. If the $X_k$ are mostly pointing in one direction, then $W_t$ will be low-rank, in which case $\alpha_t = \alpha(U_0, U_0 + W_t)\approx 1$ and $g_t\approx 0$, blowing up the bound. Thus, applying the bound this way can lead to vacuous results, at least until $\gamma_{\min}(W_t)$ grows sufficiently large. However, we may instead take $\Gamma = \eps U_0$ and $U_0' = (1-\eps)U_0$ for any $0<\eps<1$, and apply the corollary with $U_0'$ in place of $U_0$. This bounds the same quantity, since
\ifarxiv
\[V_t + U_0' = W_t + \Gamma + U_0' = W_t + U_0,\]
\else
$V_t + U_0' = W_t + \Gamma + U_0' = W_t + U_0$,
\fi
but now $\alpha_t = \alpha(U_0', U_0' + V_t) = \alpha((1-\eps)U_0, U_0 + W_t) \leq 1-\eps$, since $U_0+W_t\succeq U_0$. The parameter $\eps$ therefore buys a guaranteed bound on $\alpha_t$ at no cost in what is being bounded.

\begin{figure}[t]
    \centering
    \begin{subfigure}[t]{0.32\textwidth}
    \includegraphics[height=3.5cm]{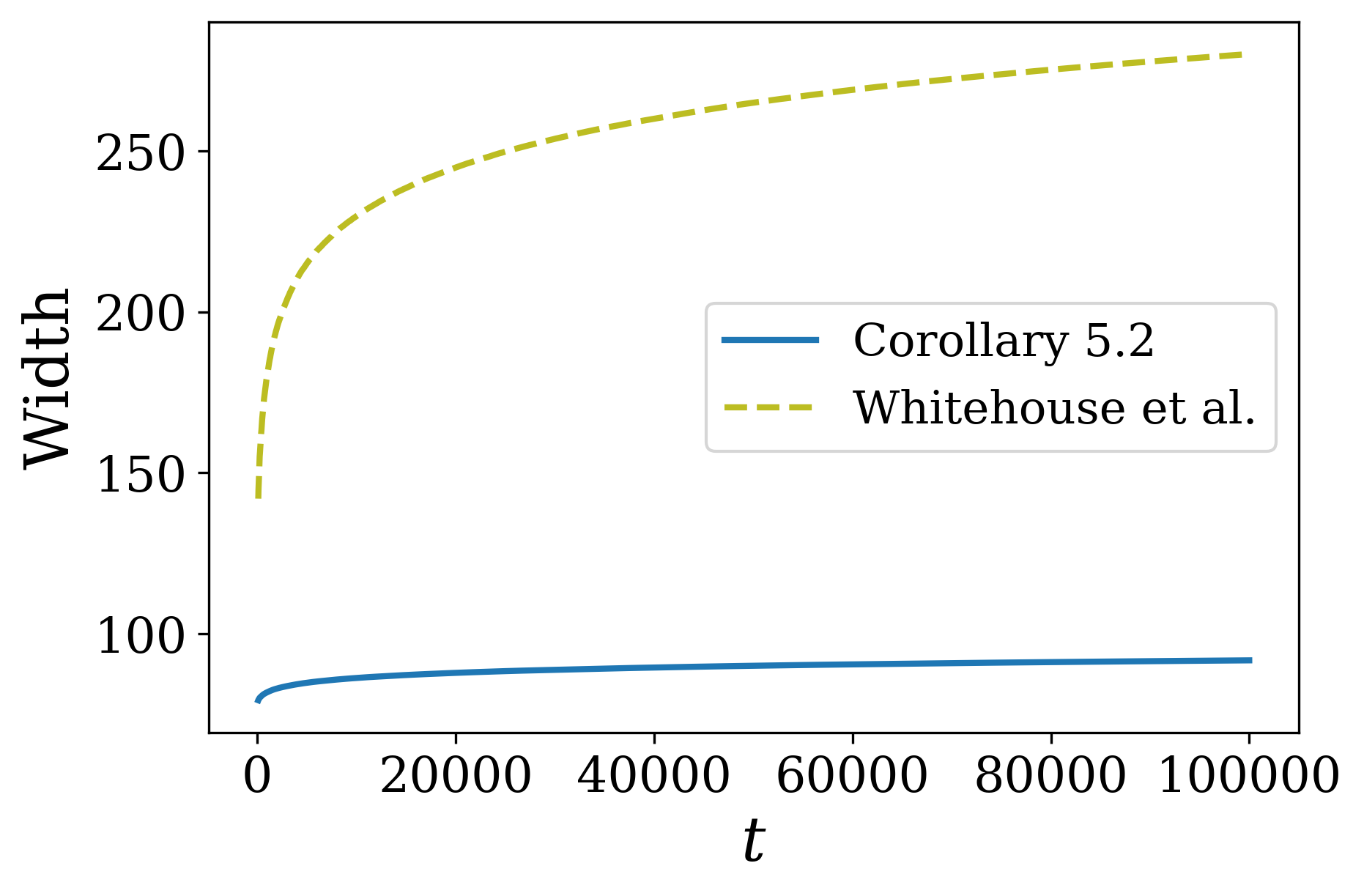}
    \caption{$k=2, \; d=20$}
    \end{subfigure}
    \begin{subfigure}[t]{0.32\textwidth}
    \includegraphics[height=3.5cm]{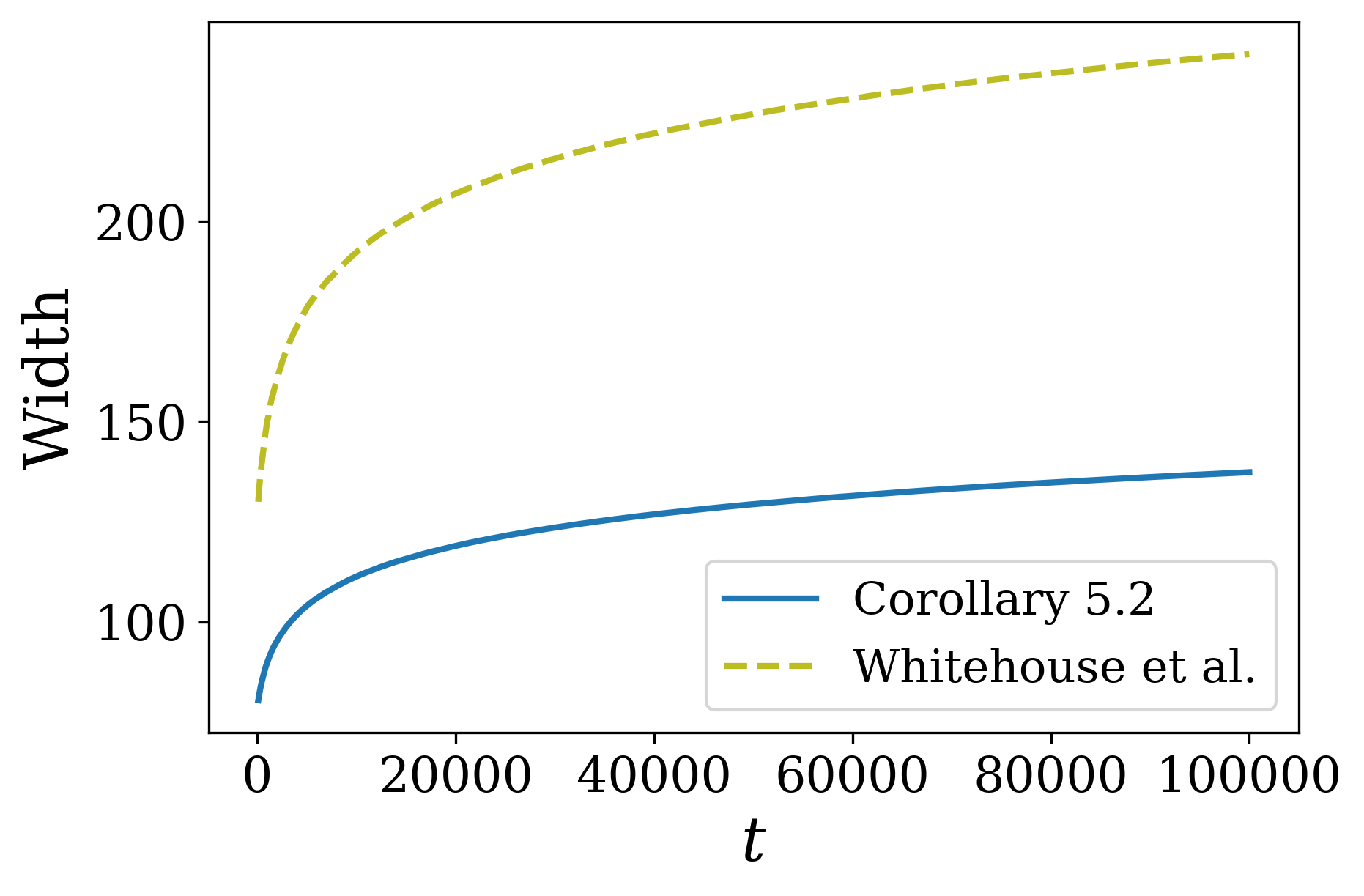}
    \caption{$k=10,\; d=20$}
    \end{subfigure}
    \begin{subfigure}[t]{0.32\textwidth}
    \includegraphics[height=3.5cm]{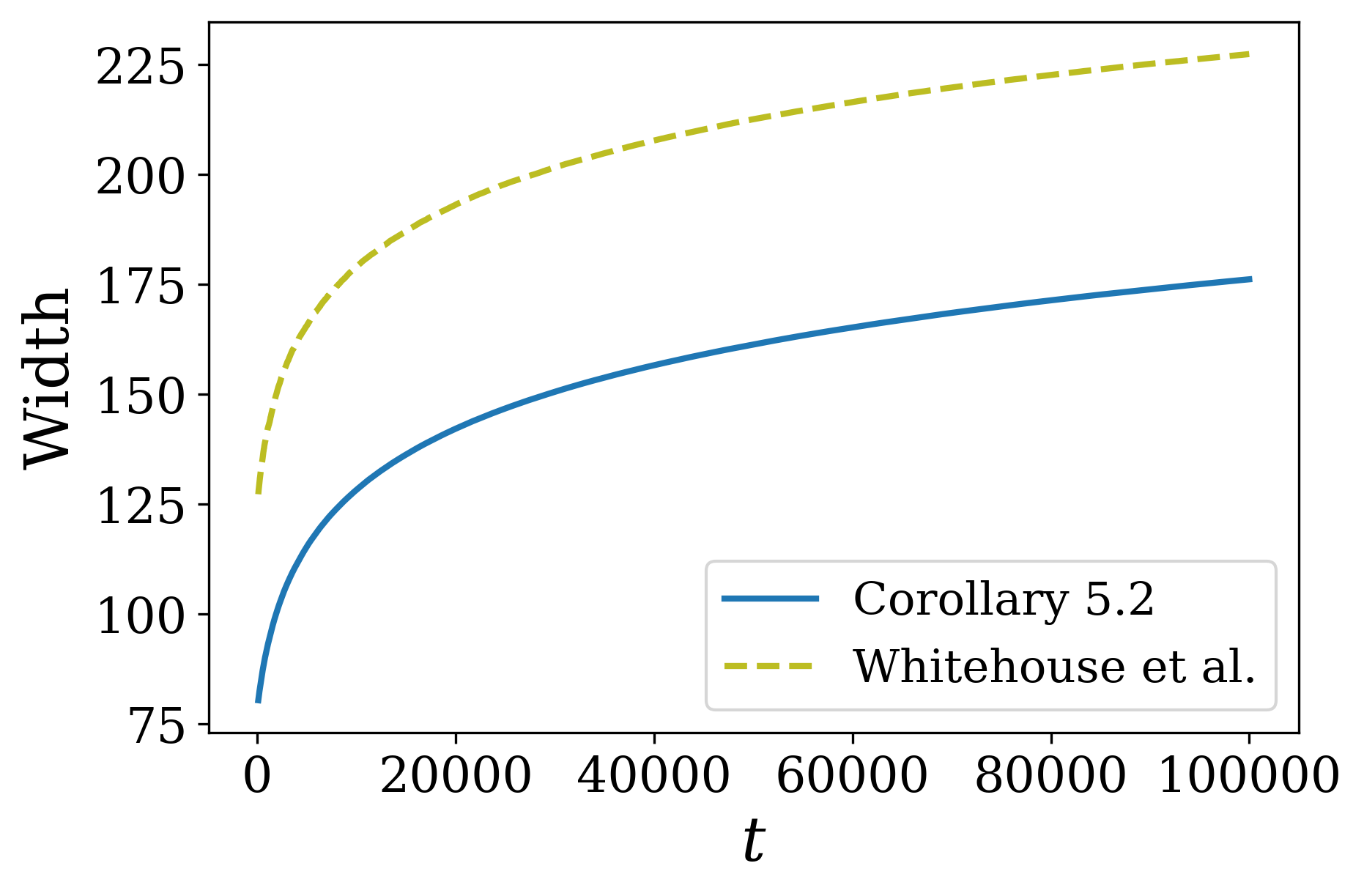}
    \caption{$k=19,\;d=20$}
    \end{subfigure}
    \caption{Comparison of our stitched empirical Bernstein bound, Corollary~\ref{cor:empirical-bernstein-stitched} (solid), with the empirical Bernstein bound of \citet{whitehouse2023time} (dashed). Here each $X_t$ is generated to have random uniform noise in $k$ directions. Each increment $(X_{t+1} - \muhat_{t})(X_{t+1} - \muhat_{t})^\intercal $ of $V_t$ is rank one, but the increments all lie in the same $k$-dimensional subspace, so $V_t$ accumulates mass in $k$ directions only. As $k$ grows, the determinant grows relative to the condition number and our advantage over the condition number-based bound narrows. Simulation details can be found in Appendix~\ref{app:experimental-details}.
    }
    \label{fig:emp-bern}
\end{figure}

In order to remove the dependence on $\lambda$ in Corollary~\ref{cor:empirical-bernstein}, we may again turn to stitching. A short calculation gives $\psi_{E,1}^*(u) = u - \log(1+u),$
so that $(\psi_{E,1}^*)^{-1}$ is the inverse of $u\mapsto u-\log(1+u)$ on $\Re_{\geq0}$, which can be written explicitly in terms of the Lambert $W$ function.

\begin{corollary}
    \label{cor:empirical-bernstein-stitched}
Let $(X_t)$ be an $\calF$-adapted stream of random vectors with conditional mean $\mu$, and suppose $\|X_t\|_2\leq 1/2$ for all $t\geq1$. Let $(\muhat_t)_{t\geq0}$ be any $\calF$-adapted sequence, let $\Gamma$ be a fixed PSD matrix, and let $V_t$ be as in~\eqref{eq:eb-process}. Let $U_0$ be positive definite with $\upsilon = \gamma_{\min}(U_0)$, and suppose that $V_t + U_0\succeq I_d$ for all $t\geq 1$. Fix $\delta\in(0,1)$.
    Then, with probability $1-\delta$, for all stopping times $\tau\geq1$,
    \begin{equation}
        \bigg\|\sum_{k\leq \tau} (X_k - \mu)\bigg\|_{(V_\tau+U_0)^{-1}} \leq \frac{2\sqrt{\upsilon}}{1-\sqrt{\alpha_\tau}}(\psi_{E,1}^*)^{-1}\left(\frac{D_\tau}{2\upsilon}\right),
    \end{equation}
    for $\alpha_\tau$ as in~\eqref{eq:alpha-t} and $D_\tau$ as in~\eqref{eq:Dtau}.
\end{corollary}
The result is an immediate application of Theorem~\ref{thm:convex_conjugate_bound} with $\psi=\psi_{E,1}$. If a closed-form expression is preferred at the cost of some looseness, $\psi_{E,1}\leq\psi_{G,1}$ implies $(\psi_{E,1}^*)^{-1}(x)\leq(\psi_{G,1}^*)^{-1}(x) = \sqrt{2x}+x$, which recovers the sub-gamma form $(1-\sqrt{\alpha_\tau})^{-1}(2\sqrt{D_\tau} + D_\tau/\sqrt{\upsilon})$ of Corollary~\ref{cor:sub-gamma-stitching}.

{Figure~\ref{fig:emp-bern} compares Corollary~\ref{cor:empirical-bernstein-stitched} to the self-normalized empirical Bernstein bound of \citet[Theorem 5.4]{whitehouse2023time}. Translated to our setting, this latter bound reads: For all $\delta\in(0,1)$, with probability $1-\delta$, for all stopping times $\tau$,
\begin{equation}
    \label{eq:whitehouse-eb-bound}
    \|S_\tau\|_{(V_\tau +U_0)^{-1}} \leq 2\sqrt{\gamma_{\min}(V_\tau +U_0)} \cdot (\psi_{E,1}^*)^{-1}\left(\frac{2 L(V_\tau)}{\gamma_{\min}(V_\tau +U_0)}\right),
\end{equation}
where
\ifarxiv
\begin{equation*}
    L(V_t) = \log\left(\frac{2}{\delta}\right) + \log\left(h\left(\log_2\left(\frac{\gamma_{\max}(V_t + U_0)}{\gamma_{\min}(U_0)}\right)\right)\right) + \log\left(2\sqrt{\kappa_t}\cdot N_{d-1}\left(\frac{1/4}{\sqrt{\kappa_t}}\right)\right),
\end{equation*}
\else
$L(V_t) = \log\left(\frac{2}{\delta}\right) + \log\left(h\left(\log_2\left(\frac{\gamma_{\max}(V_t + U_0)}{\gamma_{\min}(U_0)}\right)\right)\right) + \log\left(2\sqrt{\kappa_t}\cdot N_{d-1}\left(\frac{1/4}{\sqrt{\kappa_t}}\right)\right), $
\fi
and $\kappa_t = \kappa(V_t+ U_0)$ is the condition number of $V_t + U_0$. We have simplified the bound somewhat by fixing certain parameters that are left as variables in the original. Here $h$ is a stitching function in the sense of Section~\ref{sec:super-gaussian-stitching} and  $N_{d-1}(\eps)$ is the $\eps$-covering number of the unit sphere $\dsphere$, i.e., the minimum size of an $\eps$-cover of $\dsphere$. This is where the dimension-dependence comes from, as $N_{d-1}(\eps) = \Omega(\eps^{-(d-1)})$.

When comparing our bound with~\eqref{eq:whitehouse-eb-bound}
we take $(\muhat_t)$ to be the sequence of sample means, i.e., $\muhat_t = t^{-1}\sum_{k\leq t}X_k$.
The results echo those in Section~\ref{sec:super-gaussian}: our bound beats the condition number-based bound of Whitehouse et al.\ insofar as $\log\det(V_t)\lesssim d\log(\kappa(V_t))$, where $\kappa(V_t)$ is the condition number of $V_t$. Our bound is tighter over the whole range plotted for each of the three values of $k$, though the margin narrows as $k$ grows and the determinant grows relative to the condition number. The two bounds do cross, but only in the fully isotropic case $k=d$, where $V_t$ acquires mass in every direction and $\kappa(V_t)$ stays bounded.  There~\eqref{eq:whitehouse-eb-bound} overtakes Corollary~\ref{cor:empirical-bernstein-stitched} at around $t=1.8\times10^4$. In practice, if we do not know how the determinant will grow, we might instantiate both bounds at level $\delta/2$ and take the smaller of the two, in which case a union bound shows that the resulting boundary holds with probability $1-\delta$.}

\section{Conclusion}
\label{sec:summary}
This work has asked and answered the question: Can self-normalized, dimension-free, bounds be given for processes beyond the sub-Gaussian case? We provide such bounds for sub-$\psi$ processes, a general class of stochastic processes which includes sub-exponential, sub-Poisson, sub-gamma, and sub-Gaussian distributions. Our methods are based on the variational approach to concentration, a relatively new and (we believe) exciting approach to multivariate concentration.

Our results bridge a gap in the literature between determinant-based bounds, previously confined to sub-Gaussian or bounded observations~\citep{abbasi2011improved,akhavan2025bernstein,ziemann2025vector,metelli2025generalized}, and the dimension-dependent, condition number-based bounds of \citet{whitehouse2023time}. We show that one can enjoy both dimension-free bounds \emph{and} the distributional generality provided by Whitehouse et al.

The story does not end here. For one, as discussed in Section~\ref{sec:suboptimal}, we expect that our bounds are suboptimal, in the sense that they do not recover expected Bernstein-like behavior for $V_t = t I_d$. Can this issue be fixed?
Second, can we remove the dependence on $(1-\sqrt{\alpha_\tau})^{-1}$ in Theorem~\ref{thm:convex_conjugate_bound}? Is the factor $g_\tau$ in Theorem~\ref{thm:sub-psi} necessary, or is it simply an artifact of the proof technique? While $g_\tau$ is easy to compute, its presence makes it tougher to reason about these bounds analytically.

\ifarxiv
Third, the fact that our bounds are dimension-free suggests that they can be applied to infinite-dimensional spaces. It is typically possible to extend dimension-free results from finite dimensional spaces to infinite dimensional spaces (e.g., \citep{mollenhauer2023concentration}). While we expect this is also possible for our bounds, some of the details pose challenges. For instance, Theorem~\ref{thm:sub-psi} relies on using uniform distributions, which are not well-defined in general Hilbert spaces as there is no Lebesgue measure in infinite dimensions.

More general spaces beyond Hilbert spaces are also of interest. Several previous works have studied mean concentration in smooth Banach spaces, for instance~\citep{pinelis1994optimum,martinez2024empirical,whitehouse2024mean}. Can self-normalized results be given in such a general setting?
Can the variational technique be applied in such a setting?
\fi

%% file: appendix.tex
\section{Self-normalized Bernstein and Bennett Inequalities}
\label{sec:bennett}

In the scalar setting, two standard inequalities for light-tailed random variables are Bennett's inequality and Bernstein's inequality; see \citet[Theorem 2.9 and Theorem 2.10]{boucheron2013concentration}. The former holds for bounded random variables and the latter holds under a moment condition which has come to be known as ``Bernstein's condition.'' Both results rely on upper bounding the log-MGF with a particular $\psi$-function; $\psi_{P,c}$ in the case of Bennett's inequality and $\psi_{G,c}$ in the case of Bernstein's.  Here we provide dimension-free, self-normalized versions of these inequalities. 

Consider an $\calF$-adapted stream $(X_t)$ of random vectors taking values in $\Re^d$ with conditional mean 0. That is, $\E_{t-1}[X_t] = 0$, where we use the shorthand $\E_{t-1}[\cdot] = \E[\cdot|\calF_{t-1}]$. Note that assuming a constant \emph{conditional} mean is more general than an iid assumption. In particular, it accommodates martingale dependence, making the resulting bounds useful in bandit applications (cf. \citealp{whitehouse2023sublinear,abbasi2011improved,lattimore2020bandit}). 

Our Bernstein and Bennett inequalities will both rely on the process $(N_t^B(\theta;\lambda))$ for $\theta\in\dsphere$ and some $\lambda>0$ defined by  
    \begin{equation}
    \label{eq:nsm-mgf}
        N_t^B(\theta;\lambda) = \prod_{k\leq t}\exp\left\{ \lambda \la \theta, X_k \ra - \log \E_{k-1} \exp(\lambda \la \theta, X_k\ra)\right\},
    \end{equation}
which, as long as  the log-MGF term $\log \E_{k-1}\exp(\lambda \la \theta, X_k\ra)$ is finite, is easily seen to satisfy $\E_{t-1}N_t^B(\theta;\lambda) = N_{t-1}^B(\theta;\lambda)$ and is hence a nonnegative martingale. (We use the superscript $B$ to signal Bennett/Bernstein). Depending on the assumptions we make on $(X_t)$, we can bound the log-MGF by some function of $\E_{k-1} X_kX_k^\intercal $, which will give rise to a sub-$\psi$ process. The following lemma showcases this when the random vectors are bounded in each direction. The resulting process  will provide our self-normalized Bennett's inequality. 

\begin{lemma}
\label{lem:bennett-process}
    Let $(X_t)$ be an $\calF$-adapted stream of random vectors in $\Re^d$ with conditional mean 0 obeying $\sup_{\theta\in \dsphere} \la \theta, X_t\ra \leq b$ almost surely for some constant $b>0$. Then $S_t = \sum_{j\leq t} X_j$ and $V_t = U_0 + \sum_{j\leq t} \E_{j-1} X_jX_j^\intercal $ constitute a sub-$\psi_{P,b}$ process for any PSD matrix $U_0$. 
\end{lemma}

The proof of Lemma~\ref{lem:bennett-process} may be found in Appendix~\ref{proof:bennett-process}. 
Next we show how to obtain a sub-$\psi$ process when the random vectors obey a Bernstein condition---the inequality in~\eqref{eq:bernstein-assumption}---in each direction. We note that Bernstein's condition is more general than boundedness, in the sense that bounded random vectors satisfy~\eqref{eq:bernstein-assumption} but the converse is not true in general. The proof of the following result is in Appendix~\ref{proof:bernstein-process}.

\begin{lemma}
\label{lem:bernstein-process}
     Let $(X_t)$ be an $\calF$-adapted stream of random vectors with conditional mean 0 and $\E_{t-1}\|X_t\|_2^2<\infty$ almost surely for every $t\geq1$. Suppose that for some $c>0$ and all $\theta\in\dsphere$, $t\geq 1$,
     \begin{equation}
     \label{eq:bernstein-assumption}
         \E_{t-1}[|\la \theta, X_t\ra|^q] \leq \frac{q!c^{q-2}}{2} \E_{t-1}\la \theta, X_t \ra^2~\text{~ for all integers~}q\geq 3.
     \end{equation}
     Then $S_t = \sum_{j\leq t} X_j$ and $V_t = U_0+ \sum_{j\leq t} \E_{j-1} X_jX_j^\intercal $ constitute a sub-$\psi_{G,c}$ process for any PSD matrix $U_0$.  
\end{lemma}

Using these two results in conjunction with Theorem~\ref{thm:sub-psi} we obtain the following line crossing inequalities. Figure~\ref{fig:psi_functions} plots $\psi_{G,1}$ against $\psi_{P,1}$ to indicate how~\eqref{eq:bennett-line-crossing} scales compared to~\eqref{eq:bernstein-line-crossing}; we include $\psi_{E,1}$ in the same figure for reference, since it governs the empirical Bernstein bound of Section~\ref{sec:emp-bernstein}.

\begin{corollary}
\label{cor:bennett-and-bernstein}
Let $(X_t)$ be an $\calF$-adapted sequence of random vectors with conditional mean 0. Let $U_0$ be positive definite. 
    Suppose either that (i) $\sup_{\theta\in\dsphere} \la \theta, X_t\ra\leq b$ for all $t\geq 1$ and some $b>0$ or (ii) that~\eqref{eq:bernstein-assumption} holds and $\E_{t-1}\|X_t\|_2^2<\infty$ almost surely for every $t\geq1$. Let $S_t = \sum_{k\leq t} X_k$ and $V_t = \sum_{k\leq t} \E_{k-1}[X_kX_k^\intercal ]$. Write $\upsilon = \gamma_{\min}(U_0)$ and let $g_\tau$ and $\Delta_\tau(\delta)$ be as in~\eqref{eq:gt} and~\eqref{eq:Delta-t}. Fix any $\delta\in(0,1)$. 
    If (i) holds, then for any $\lambda \in(0,\infty)$, with probability $1-\delta$, for any stopping time $\tau$, 
    \begin{equation}
    \label{eq:bennett-line-crossing}
         \|S_\tau\|_{(V_\tau+U_0)^{-1}} \leq \frac{\Delta_\tau(\delta) + 2\upsilon\psi_{P,b}(\lambda)}{\lambda g_\tau}.
    \end{equation}
    If (ii) holds, then for any $\lambda\in(0,1/c)$, with probability $1-\delta$, for any stopping time $\tau$, 
    \begin{equation}
    \label{eq:bernstein-line-crossing}
         \|S_\tau\|_{(V_\tau+U_0)^{-1}} \leq \frac{\Delta_\tau(\delta) + 2\upsilon\psi_{G,c}(\lambda)}{\lambda g_\tau}.
    \end{equation}
\end{corollary}
Since the process is sub-gamma when $(X_t)$ satisfies the Bernstein condition in~\eqref{eq:bernstein-assumption}, we can in that case apply Corollary~\ref{cor:sub-gamma-stitching} to obtain the stitched boundary in~\eqref{eq:sub-gamma-bound}. 
In particular, by carefully applying the line-crossing inequality above with different values of $\lambda$ at different times, we obtain 
$\|S_\tau\|_{(V_\tau + U_0)^{-1}}$ scaling like $\sqrt{D_\tau} + cD_\tau/\sqrt{\upsilon}$ with high probability, for $D_\tau$ as in~\eqref{eq:Dtau}. The combination of a square-root and a linear term in $D_\tau$ is what we expect of a Bernstein bound; as discussed in Section~\ref{sec:suboptimal}, however, we suspect that the scale $\sqrt{\upsilon}$ normalizing the linear term is suboptimal.

\begin{figure}
    \centering
    \includegraphics[width=0.35\textwidth]{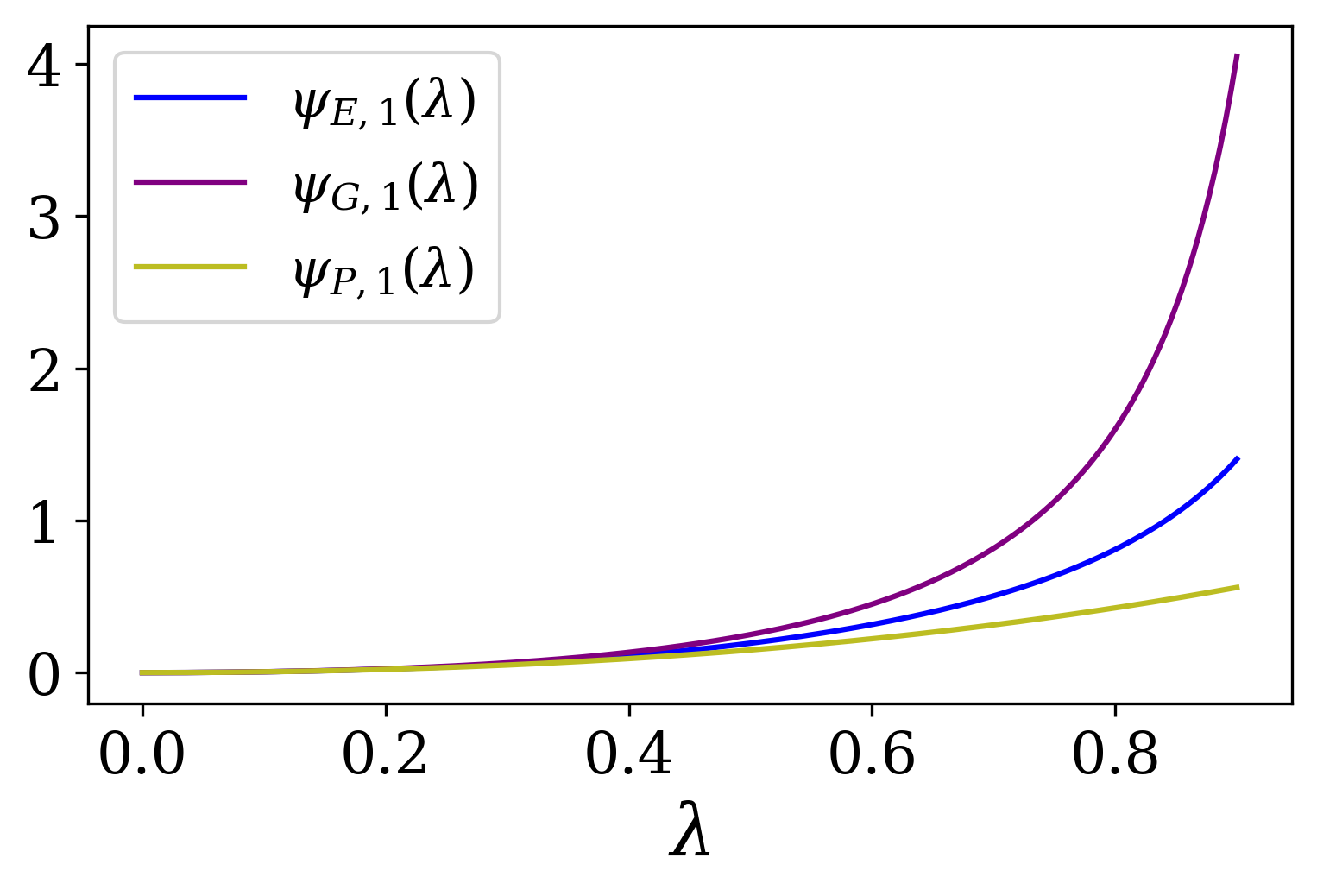}
    \caption{
    Comparison of $\psi_{E,1}$,  $\psi_{G,1}$, and $\psi_{P,1}$ for $\lambda\in[0,1)$. }
    \label{fig:psi_functions}
\end{figure}

\section{Omitted Proofs}
\label{app:proofs}

\subsection{Proof of Theorem~\ref{thm:sub-gaussian}}
\label{proof:sub-gaussian}
Let $(S_t)$ be a sub-Gaussian process with variance proxy $V_t$. For each $\lambda\geq 0$ and $\theta\in\Re^d$, let
\begin{equation}
\label{eq:Zt-subG}
    Z_t^\lambda(\theta) = \exp\big\{ \lambda \la \theta, S_t\ra - \psi_N(\lambda) \la \theta, V_t\theta\ra\big\}.
\end{equation}
For $\theta\in\dsphere$, $(Z_t(\theta))$ is upper-bounded by a nonnegative supermartingale by Definition~\ref{def:sub-psi}. We claim that this property extends to all $\theta\in\Re^d$. The claim is immediate for $\theta=0$, since $Z_t^\lambda(0)=1$ for every $t$. For nonzero $\theta$, let $\phi = \theta /\|\theta\|_2$. Note that $\phi\in\dsphere$. Observe that
\begin{align*}
    Z_t^\lambda(\theta) = \exp\{ \lambda \|\theta\| \la \phi, S_t \ra - \psi_N(\lambda \|\theta\|) \la \phi, V_t\phi\ra\} = Z_t^{\lambda \|\theta\|}(\phi),
\end{align*}
which is upper bounded by $L_t^{\lambda \|\theta\|}(\phi)$ (note that $\lambda_{\max} = \infty$ in this case). Thus, the process defined by~\eqref{eq:Zt-subG} is upper bounded by a nonnegative supermartingale for all $\theta\in\Re^d$, enabling us to apply Proposition~\ref{prop:variational_template} with $\Theta=\Re^d$.

Let $\tau$ be a stopping time and consider a Gaussian distribution $\rho = \rho_\tau$ with mean $\mu_\rho$ and covariance $\Sigma_\rho$. We have
\begin{align}
\int\log Z_\tau(\theta) \d\rho &=  \lambda \la \mu_\rho, S_\tau\ra - \psi_N(\lambda) \int \la \theta, V_\tau\theta \ra \d\rho  \notag \\
&= \lambda \la \mu_\rho, S_\tau\ra - \psi_N(\lambda)\left( \la \mu_\rho, V_\tau\mu_\rho\ra  +  \Tr(V_\tau \Sigma_\rho)\right).
\end{align}
Define $\hatV_t = V_t + U_0$ and set $\mu_\rho = \hatV_\tau^{-1}S_\tau$. Then
\begin{align*}
    \int\log Z^\lambda_\tau(\theta) \rho(\d\theta) &= \lambda\|S_\tau\|_{\hatV_\tau^{-1}}^2 - \psi_N(\lambda) \|S_\tau\|_{\hatV_\tau^{-1}V_\tau\hatV_\tau^{-1}}^2 - \psi_N(\lambda) \Tr(V_\tau \Sigma_\rho) \\
    &= (\lambda  - \psi_N(\lambda))\|S_\tau\|_{\hatV_\tau^{-1}}^2 +  \psi_N(\lambda) \|S_\tau\|^2_{\hatV_\tau^{-1} U_0 \hatV_\tau^{-1}} - \psi_N(\lambda) \Tr(V_\tau\Sigma_\rho).
\end{align*}
Let $\nu$ be a Gaussian with mean 0 and covariance $\Sigma_\nu$. Then
\begin{align*}
    \kl(\rho\|\nu) &= \frac{1}{2}\left\{\Tr(\Sigma_\nu^{-1}\Sigma_\rho) + \la \mu_\rho, \Sigma_\nu^{-1}\mu_\rho\ra -d + \log\left(\frac{\det \Sigma_\nu}{\det \Sigma_\rho}\right)\right\} \\
    &= \frac{1}{2}\left\{\Tr(\Sigma_\nu^{-1}\Sigma_\rho) +  \|S_\tau\|^2_{\hatV_\tau^{-1} \Sigma_\nu^{-1} \hatV_\tau^{-1}} -d + \log\left(\frac{\det \Sigma_\nu}{\det \Sigma_\rho}\right)\right\}.
\end{align*}
In order to match the term $\psi_N(\lambda) \|S_\tau\|^2_{\hatV_\tau^{-1} U_0 \hatV_\tau^{-1}}$ above, choose $\Sigma_\nu = (2\psi_N(\lambda) U_0)^{-1}$. Proposition~\ref{prop:variational_template} implies that with probability $1-\delta$,
\[
    \int \log Z_\tau(\theta) \d\rho \leq \kl(\rho\|\nu) + \log(1/\delta),
\]
which rearranges to
\begin{align}
\label{eq:subG_bound1}
   (\lambda  - \psi_N(\lambda))\|S_\tau\|_{\hatV_\tau^{-1}}^2 \leq \psi_N(\lambda)\Tr(\hatV_\tau \Sigma_\rho)  - \frac{d}{2} + \frac{1}{2}\log\left(\frac{\det \Sigma_\nu}{\det \Sigma_\rho}\right) + \log(1/\delta).
\end{align}
Taking $\Sigma_\rho = (2\psi_N(\lambda) \hatV_\tau)^{-1}$ gives $\psi_N(\lambda) \Tr(\hatV_\tau\Sigma_\rho) = d/2$ and $\det \Sigma_\nu / \det \Sigma_\rho = \det \hatV_\tau / \det U_0$. Assuming $\lambda - \psi_N(\lambda)>0$, \eqref{eq:subG_bound1} thus becomes
\begin{equation}
    \|S_\tau\|_{(V_\tau + U_0)^{-1}}^2 \leq \frac{1}{\lambda - \psi_N(\lambda)}\left(\frac{1}{2}\log\left(\frac{\det (V_\tau + U_0) }{\det  U_0 }\right) + \log(1/\delta)\right).
\end{equation}
Finally, we optimize over $\lambda$ to achieve the maximum value of $\lambda^* - \psi_N(\lambda^*) = 1/2$ at $\lambda^*=1$. This completes the proof.

\subsection{Proof of Theorem~\ref{thm:sub-psi}}
\label{proof:sub-psi}

When $g_t = 0$ we interpret the right hand side of the bound as $+\infty$ so the result is vacuous. 

Note that if $(S_t,V_t)$ is a sub-$\psi$ process, then so too is $(S_t,V_t+U_0)$ for any PSD matrix $U_0$. Let $\hatV_t=V_t+U_0.$
For the remainder of the proof we consider the process $(S_t,\hatV_t)$ and set
\begin{equation}
    Z_t(\theta)
    =
    \exp\left\{
        \lambda\la\theta,S_t\ra
        -
        \psi(\lambda)\la\theta,\hatV_t\theta\ra
    \right\}.
\end{equation}

For any $\theta\in\dsphere$, the process $(Z_t(\theta))$ is upper bounded by a nonnegative supermartingale by assumption. We claim that the same holds for all $\theta\in\dball$. The claim is immediate for $\theta=0$. For nonzero $\theta\in\dball$, let $\phi=\theta/\|\theta\|$. Then
\begin{align*}
    \exp\left\{
        \lambda\la\theta,S_t\ra
        -
        \psi(\lambda)\la\theta,\hatV_t\theta\ra
    \right\} 
    &=
    \exp\left\{
        \lambda\|\theta\|\la\phi,S_t\ra
        -
        \psi(\lambda)\|\theta\|^2
        \la\phi,\hatV_t\phi\ra
    \right\} \\
    &\leq
    \exp\left\{
        \lambda\|\theta\|\la\phi,S_t\ra
        -
        \psi(\lambda\|\theta\|)
        \la\phi,\hatV_t\phi\ra
    \right\},
\end{align*}
where the inequality follows because $\lambda\mapsto\psi(\lambda)/\lambda^2$ is nondecreasing. Since $\lambda\|\theta\|<\lambda_{\max}$, the final quantity is upper bounded by a nonnegative supermartingale. We may therefore take $\Theta=\dball$ in Proposition~\ref{prop:variational_template}.

Now, proceeding to the main proof, let $\tau$ be a stopping time and let $\rho=\rho_\tau$ be a distribution over $\dball$ with mean $\mu_\rho$ and covariance $\Sigma_\rho$. Then
\begin{align}
\int\log Z_\tau(\theta)\d\rho
&=
\lambda\la\mu_\rho,S_\tau\ra
-
\psi(\lambda)
\int\la\theta,\hatV_\tau\theta\ra\d\rho
\notag\\
&=
\lambda\la\mu_\rho,S_\tau\ra
-
\psi(\lambda)
\left(
    \la\mu_\rho,\hatV_\tau\mu_\rho\ra
    +
    \Tr(\hatV_\tau\Sigma_\rho)
\right).
\label{eq:int_log_Zt1}
\end{align}
Write $\upsilon=\gamma_{\min}(U_0)$. For some $\beta>0$ to be chosen later, let $\rho$ be uniform over the ellipsoid with mean $
\mu_\rho=\beta\hatV_\tau^{-1}S_\tau$
and shape $A_\rho=\upsilon\hatV_\tau^{-1}.$
That is, $\rho$ is uniform over $\calE_\rho=\left\{
        \theta\in\Re^d:
        (\theta-\mu_\rho)^\intercal
        A_\rho^{-1}
        (\theta-\mu_\rho)
        \leq 1
    \right\},$
which we will ensure is a subset of $\dball$. The distribution $\rho$ has covariance
\[
    \Sigma_\rho=\frac{1}{d+2}A_\rho
    =
    \frac{\upsilon}{d+2}\hatV_\tau^{-1}.
\]
Consequently,
\begin{equation*}
    \psi(\lambda)\Tr(\hatV_\tau\Sigma_\rho)
    =
    \frac{d}{d+2}\upsilon\psi(\lambda),
\end{equation*}
and substituting this into~\eqref{eq:int_log_Zt1} gives
\begin{align}
\int\log Z_\tau(\theta)\d\rho
&=
\left(
    \lambda\beta-\beta^2\psi(\lambda)
\right)
\|S_\tau\|_{\hatV_\tau^{-1}}^2
-
\frac{d}{d+2}\upsilon\psi(\lambda).
\end{align}

Proposition~\ref{prop:variational_template} therefore implies that, with probability $1-\delta$,
\begin{align}
\left(
    \lambda\beta-\beta^2\psi(\lambda)
\right)
\|S_\tau\|_{\hatV_\tau^{-1}}^2
&\leq
\kl(\rho\|\nu)
+
\frac{d}{d+2}\upsilon\psi(\lambda)
+
\log(1/\delta).
\label{eq:bound1}
\end{align}

Take $\nu$ to be uniform over the ellipsoid
$\calE_\nu
=
\left\{
    \theta\in\Re^d:
    \theta^\intercal A_\nu^{-1}\theta\leq 1
\right\}$ where 
$A_\nu=\upsilon U_0^{-1}.$
Since $A_\nu^{-1}
=
\upsilon^{-1}U_0
\succeq I_d$, 
we have $\calE_\nu\subset\dball$. If $\beta$ is small enough that $\calE_\rho\subset\calE_\nu$, then
\begin{align*}
\kl(\rho\|\nu)
&=
\int
\log\left(
    \frac{\d\rho}{\d\nu}(\theta)
\right)
\d\rho 
=
\log\left(
    \frac{\operatorname{vol}(\calE_\nu)}
         {\operatorname{vol}(\calE_\rho)}
\right) =
\log\left(
    \frac{\sqrt{\det A_\nu}}
         {\sqrt{\det A_\rho}}
\right) =
\frac{1}{2}
\log\left(
    \frac{\det\hatV_\tau}{\det U_0}
\right).
\end{align*}
It remains to choose $\beta$ so that $\calE_\rho\subset\calE_\nu$. Put $y=y_\tau=\hatV_\tau^{-1}S_\tau.$
By shifting everything, 
it suffices to show that $\theta^\intercal \hatV_\tau\theta\leq\upsilon$
implies $(\theta+\beta y)^\intercal U_0(\theta+\beta y)
\leq\upsilon.$
For any $\eps>0$,
\begin{align*}
(\theta+\beta y)^\intercal U_0(\theta+\beta y)
&=
\theta^\intercal U_0\theta
+
2\beta\theta^\intercal U_0y
+
\beta^2y^\intercal U_0y \\
&\leq
(1+\eps)\theta^\intercal U_0\theta
+
\beta^2(1+\eps^{-1})y^\intercal U_0y  && \text{Young's inequality} \\
&\leq
(1+\eps)\alpha_\tau\theta^\intercal \hatV_\tau\theta
+
\beta^2(1+\eps^{-1})y^\intercal U_0y && \text{Rayleigh coefficient}\\
&\leq
(1+\eps)\alpha_\tau\upsilon
+
\beta^2(1+\eps^{-1})\|y\|_{U_0}^2
=:h(\eps,\beta). && \theta^\intercal \hatV_\tau \theta\leq \upsilon 
\end{align*}
Minimizing over $\eps>0$ gives
\[
    \eps^*
    =
    \frac{\beta\|y\|_{U_0}}
         {\sqrt{\alpha_\tau\upsilon}},
\]
and hence $
    h(\eps^*,\beta)
    =
    \left(
        \beta\|y\|_{U_0}
        +
        \sqrt{\alpha_\tau\upsilon}
    \right)^2.$
Thus, the largest value of $\beta$ certified by this argument is
\begin{equation}
\label{eq:beta_star}
    \beta^*
    =
    \frac{
        \sqrt{\upsilon}
        \left(1-\sqrt{\alpha_\tau}\right)
    }{
        \|y\|_{U_0}
    }
    =
    \frac{g_\tau}{\|y\|_{U_0}}.
\end{equation}

Moreover,
\begin{align*}
\|y\|_{U_0}^2
&=
S_\tau^\intercal
\hatV_\tau^{-1}
U_0
\hatV_\tau^{-1}
S_\tau =
(S_\tau^\intercal \hatV_\tau^{-1/2})
(\hatV_\tau^{-1/2}U_0\hatV_\tau^{-1/2})
(\hatV_\tau^{-1/2}S_\tau) \\
&\leq
\|\hatV_\tau^{-1/2}U_0\hatV_\tau^{-1/2}\|
\,
S_\tau^\intercal \hatV_\tau^{-1}S_\tau 
\leq
\|S_\tau\|_{\hatV_\tau^{-1}}^2,
\end{align*}
where the final inequality follows from $U_0\preceq\hatV_\tau$. If $\|S_\tau\|_{\hatV_\tau^{-1}}=0$, the result is immediate. Otherwise, set
\[
    \beta^{**}
    =
    \frac{g_\tau}
         {\|S_\tau\|_{\hatV_\tau^{-1}}}
    \leq
    \beta^*.
\]
Since $h(\eps^*,\beta)$ is increasing in $\beta$, the containment certified at $\beta^*$ also holds at $\beta^{**}$. It follows that $\calE_\rho\subset\calE_\nu$, so inequality~\eqref{eq:bound1} holds with $\beta=\beta^{**}$. With this choice,
\begin{align*}
\left(
    \lambda\beta^{**}
    -
    (\beta^{**})^2\psi(\lambda)
\right)
\|S_\tau\|_{\hatV_\tau^{-1}}^2 
=
\lambda g_\tau
\|S_\tau\|_{\hatV_\tau^{-1}}
-
g_\tau^2\psi(\lambda).
\end{align*}
Therefore, recalling from~\eqref{eq:gt} that $g_\tau^2=\upsilon(1-\sqrt{\alpha_\tau})^2$,
\begin{align*}
\lambda g_\tau
\|S_\tau\|_{\hatV_\tau^{-1}}
&\leq
\frac{1}{2}
\log\left(
    \frac{\det\hatV_\tau}{\det U_0}
\right)
+
\log(1/\delta)
+
\frac{d}{d+2}\upsilon\psi(\lambda)
+
g_\tau^2\psi(\lambda) \\
&\leq
\frac{1}{2}
\log\left(
    \frac{\det\hatV_\tau}{\det U_0}
\right)
+
\log(1/\delta)
+
\left(
    1+\left(1-\sqrt{\alpha_\tau}\right)^2
\right)
\upsilon\psi(\lambda),
\end{align*}
where the second inequality uses $d/(d+2)\leq 1$. Dividing by $\lambda g_\tau$ gives the sharper bound~\eqref{eq:sub-psi-bound-sharp} of Remark~\ref{rem:sharper-constant}. Finally, $\alpha_\tau\in[0,1)$ implies $(1-\sqrt{\alpha_\tau})^2\leq 1$, which proves~\eqref{eq:sub-psi-bound}.

\subsection{Proof of Theorem~\ref{thm:convex_conjugate_bound}}
\label{proof:convex_conjugate}

Set $\hatV_t=V_t+U_0$ and $\upsilon=\gamma_{\min}(U_0)$. Recall that $g_t=\sqrt{\upsilon}\left(1-\sqrt{\alpha_t}\right)$.
We divide time into the epochs
\[
    E_k
    =
    \left\{
        t:
        2^k\leq\det\hatV_t<2^{k+1}
    \right\},
    \qquad k\geq0.
\]
Since $\hatV_t\succeq I_d$ we have $\det\hatV_t\geq1$, so the epochs $(E_k)_{k\geq0}$ form a bona fide partition of $\{t\geq1\}$.
Let
\[
    \ell(k)=\frac{\pi^2}{6}(k+1)^2,
    \qquad
    \delta_k=\frac{\delta}{\ell(k)}.
\]
(Here function $\ell$ is called a ``stitching function.'' Different functions $\ell$ satisfying $\sum_k \ell^{-1}(k)\leq 1$ can be chosen, which results in different constants in the final bound. For simplicity, we stick to a particular choice of $\ell$ here. See \citet{howard2021time} for more detail on stitching functions and how they affect the final boundary shape.)
Then $\sum_{k\geq0}\delta_k=\delta$. Since $\hatV_t\succeq U_0$ implies $\det\hatV_t\geq\det U_0$, the epoch $E_k$ is empty unless $2^{k+1}>\det U_0$, and we restrict attention to such $k$ from here on. For each of them, define
\begin{equation}
\label{eq:Ck-convex-stitching}
    L_k
    =
    \frac{1}{2}
    \log\left(
        \frac{2^{k+1}}{\det U_0}
    \right)
    +
    \log(1/\delta_k)
    >0,
    \qquad
    C_k=L_k+\frac{3}{20},
\end{equation}
and set
\begin{equation*}
    x_k=\frac{C_k}{2\upsilon},
    \qquad
    y_k=(\psi^*)^{-1}(x_k).
\end{equation*}
The conjugate $\psi^*$ is finite on $\Re_{\geq0}$: this is immediate if $\lambda_{\max}<\infty$, since then $\psi^*(u)\leq\lambda_{\max}u$, and if $\lambda_{\max}=\infty$ it follows from $\psi(\lambda)\geq\psi(1)\lambda^2$ for $\lambda\geq1$, a consequence of super-Gaussianity. Being finite and convex, $\psi^*$ is therefore continuous on $(0,\infty)$, and it is nondecreasing and unbounded since $\psi^*(u)\geq\lambda u-\psi(\lambda)\to\infty$ as $u\to\infty$ for any fixed $\lambda\in(0,\lambda_{\max})$. Consequently $0<y_k<\infty$ and $\psi^*(y_k)\geq x_k$. The role of the constant $3/20$ in~\eqref{eq:Ck-convex-stitching} is to make this inequality strict relative to $L_k$: since $\upsilon>0$,
\[
    \psi^*(y_k)\geq x_k>\frac{L_k}{2\upsilon}.
\]
By the definition of the convex conjugate (recall that $\psi^*(y)=\sup_\lambda\{\lambda y-\psi(\lambda)\}$) there is therefore a deterministic $\lambda_k\in(0,\lambda_{\max})$ with $\lambda_ky_k-\psi(\lambda_k)\geq L_k/(2\upsilon)$, that is,
\begin{equation}
\label{eq:convex-duality-choice}
    L_k+2\upsilon\psi(\lambda_k)\leq2\upsilon\lambda_ky_k.
\end{equation}
Note that we do not need the supremum defining $\psi^*(y_k)$ to be attained. For $t\in E_k$, let
\[
    K_{t,k}
    =
    \frac{1}{2}
    \log\left(
        \frac{\det\hatV_t}{\det U_0}
    \right)
    +
    \log(1/\delta_k).
\]
By definition of the epoch, $K_{t,k}\leq L_k$. Applying Theorem~\ref{thm:sub-psi} with confidence level $\delta_k$ and parameter $\lambda_k$ gives, with probability at least $1-\delta_k$, simultaneously for all $t\in E_k$,
\begin{align*}
    \|S_t\|_{\hatV_t^{-1}}
    &\leq
    \frac{K_{t,k}+2\upsilon\psi(\lambda_k)}
         {\lambda_k g_t}
    \leq
    \frac{L_k+2\upsilon\psi(\lambda_k)}
         {\lambda_k g_t}
    \leq
    \frac{2\upsilon y_k}{g_t}
    =
    \frac{2\sqrt{\upsilon}}
         {1-\sqrt{\alpha_t}}
    (\psi^*)^{-1}\left(\frac{C_k}{2\upsilon}\right),
\end{align*}
where the third inequality is~\eqref{eq:convex-duality-choice} and the equality uses $g_t=\sqrt{\upsilon}(1-\sqrt{\alpha_t})$.
A union bound over $k$ shows that the preceding inequality holds simultaneously for all $t\geq1$ with probability at least $1-\delta$. It remains only to replace $C_{k(t)}$ by an  upper bound. If $t\in E_k$, then $2^{k+1}\leq2\det\hatV_t$, and $k+1\leq\log_2(\det\hatV_t)+1$. 
Hence
\begin{align*}
    C_{k(t)}
    &\leq
    \frac{1}{2}
    \log\left(
        \frac{\det\hatV_t}{\det U_0}
    \right)
    +
    \frac{1}{2}\log 2
    +
    \log\left(\frac{\pi^2}{6}\right)
    +
    \frac{3}{20}
    +
    2\log\left(
        \log_2(\det\hatV_t)+1
    \right)
    +
    \log(1/\delta)
    \leq D_t,
\end{align*}
where
\[
    D_t
    =
    \frac{1}{2}
    \log\left(
        \frac{\det\hatV_t}{\det U_0}
    \right)
    +
    1
    +
    2\log\left(
        \log_2(\det\hatV_t)+1
    \right)
    +
    \log(1/\delta),
\]
and the final inequality uses that $
\frac{1}{2}\log 2+\log\left(\frac{\pi^2}{6}\right)+\frac{3}{20}<1.$
Since $(\psi^*)^{-1}$ is nondecreasing, we conclude that, with probability at least $1-\delta$, for every $t\geq1$,
\[
    \|S_t\|_{\hatV_t^{-1}}
    \leq
    \frac{2\sqrt{\upsilon}}
         {1-\sqrt{\alpha_t}}
    (\psi^*)^{-1}\left(\frac{D_t}{2\upsilon}\right).
\]
The equivalent stopping-time statement follows immediately.

\subsection{Proof of Corollary~\ref{cor:sub-gamma-stitching}}
\label{proof:sub-gamma}

Substituting the identity~\eqref{eq:sub-gamma-conjugate-inverse} into Theorem~\ref{thm:convex_conjugate_bound} yields
\begin{align*}
    \|S_\tau\|_{(V_\tau+U_0)^{-1}}
    &\leq
    \frac{2\sqrt{\upsilon}}
         {1-\sqrt{\alpha_\tau}}
    \left(
        \sqrt{\frac{D_\tau}{\upsilon}}
        +
        \frac{cD_\tau}{2\upsilon}
    \right) 
    =
    \frac{1}{1-\sqrt{\alpha_\tau}}
    \left(
        2\sqrt{D_\tau}
        +
        \frac{cD_\tau}{\sqrt{\upsilon}}
    \right),
\end{align*}
which is~\eqref{eq:sub-gamma-bound}.

\subsection{Proof of Lemma~\ref{lem:bennett-process}}
\label{proof:bennett-process}
Our goal is to upper bound the term $\log \E_{t-1} \exp(\lambda\la \theta, X_t\ra)$ in~\eqref{eq:nsm-mgf}. Since $\log x \leq x - 1$ for all $x>0$, we have
    \[\log \E_{t-1} \exp(\lambda \la \theta, X_t\ra) \leq \E_{t-1}\exp(\lambda \la \theta, X_t\ra) - 1 = \E_{t-1}\psi_{P,1}(\lambda \la \theta, X_t\ra),\]
    where the final equality follows since $\E_{t-1}\la \theta, X_t\ra = 0$ (recall that $\psi_{P,1}(x) = e^x - x - 1$). Now, notice that $f(c) = \psi_{P,c}(x)$ is an increasing function of $c$ (not of $x$!). Since $\la \theta, X_t\ra\leq  b$ almost surely, we have
    \begin{equation*}
        \frac{e^{\lambda\la \theta, X_t\ra} - \lambda \la \theta, X_t\ra  -1}{\la \theta, X_t\ra^2} \leq \psi_{P,b}(\lambda ).
    \end{equation*}
    Rearranging and taking expectations gives
    \begin{equation*}
        \E_{t-1}\psi_{P,1}(\lambda \la \theta, X_t\ra) \leq \E_{t-1}[ \la \theta, X_t\ra^2 ] \psi_{P,b}(\lambda ).
    \end{equation*}
    This proves that
    \[N_t^{\text{Ben}}(\theta) = \prod_{k\leq t} \exp\big\{ \lambda\la \theta, X_k\ra - \psi_{P,b}(\lambda)\E_{k-1}[\la \theta, X_k\ra^2]\big\} \leq N_t^B(\theta),\]
    hence $(N_t^{\text{Ben}}(\theta))$ is upper bounded by a nonnegative martingale. Noticing that $N_t^{\text{Ben}}(\theta) = \exp\{\lambda \la \theta, S_t\ra - \psi_{P,b}(\lambda)\la \theta, V_t\theta\ra\}$ proves that $(S_t,V_t)$ is a sub-$\psi$ process for $S_t = \sum_{k\leq t} X_k$, $V_t = \sum_{k\leq t} \E_{k-1}X_kX_k^\intercal $. Adding any PSD matrix $U_0$ to $V_t$ can only make the process smaller, which completes the proof.

\subsection{Proof of Lemma~\ref{lem:bernstein-process}}
\label{proof:bernstein-process}
In the proof of Lemma~\ref{lem:bennett-process}, we showed that $\log\E_{t-1}\exp(\lambda\la \theta, X_t\ra) \leq \E_{t-1}\psi_{P,1}(\lambda\la \theta, X_t\ra)$.
Using the Taylor expansion $e^x = \sum_{q\geq 0} x^q/q!$ it follows that $\psi_{P,1}(x) = \sum_{q\geq 2} \frac{x^q}{q!}$. Assumption~\eqref{eq:bernstein-assumption} bounds the absolute moments of $\la \theta, X_t\ra$, so for $\lambda<1/c$ the series $\sum_{q\geq2}\lambda^q\E_{t-1}[|\la \theta, X_t\ra|^q]/q!$ converges and we may exchange expectation and summation:
    \begin{align*}
      \E_{t-1} \psi_{P,1}(\lambda \la \theta, X_t\ra) & = \sum_{q\geq 2} \frac{ \lambda^q\E_{t-1}[\la \theta, X_t\ra^q]}{q!}.
    \end{align*}
    Using assumption~\eqref{eq:bernstein-assumption} again, together with $\E_{t-1}[\la \theta, X_t\ra^q] \leq \E_{t-1}[|\la \theta, X_t\ra|^q]$, we have
    \begin{align*}
        \sum_{q\geq 2} \frac{ \lambda^q\E_{t-1}[\la \theta, X_t\ra^q]}{q!}
      &\leq \sum_{q\geq 2} \frac{\lambda^q c^{q-2}\E_{t-1}[\la \theta, X_t\ra^2]}{2} \leq \psi_{G,c}(\lambda)\E_{t-1}[\la \theta, X_t\ra^2],
    \end{align*}
    if $\lambda<1/c$. This implies that the process
    \[N_t^{\text{Bern}}(\theta) = \prod_{k\leq t} \exp\left\{\lambda \la\theta, X_k\ra - \psi_{G,c}(\lambda)\la \theta, \E_{k-1}X_kX_k^\intercal \theta\ra\right\},\]
    is upper bounded by the nonnegative martingale $N_t^B(\theta)$ of~\eqref{eq:nsm-mgf}, which is all that Definition~\ref{def:sub-psi} requires. As in the proof of Lemma~\ref{lem:bennett-process}, noticing that we can rewrite $N_t^{\text{Bern}}(\theta)$ as $N_t^{\text{Bern}}(\theta) = \exp\{\lambda \la \theta, S_t\ra - \psi_{G,c}(\lambda) \la \theta, V_t\theta\ra \}$ for $V_t = \sum_{k\leq t} \E_{k-1}X_kX_k^\intercal $ shows that $(S_t,V_t)$ is sub-$\psi_{G,c}$, and adding any PSD matrix to $V_t$ maintains this property.

\section{Simulation Details}
\label{app:experimental-details}

Code may be found at \href{https://github.com/bchugg/sn-concentration}{https://github.com/bchugg/sn-concentration}.

\paragraph{Figure~\ref{fig:subgamma}.}
We take $U_0 = I_d$ and $d=20$ and $c=1$.  For the first 100 time steps, we make a full rank update in order to sufficiently grow the minimum eigenvalue. These updates are dampened based on the size of the eigenvalue, however, so there is still anisotropic growth of the matrix. The dampening factor is 1 for the top eigenvector (i.e., corresponding to the largest eigenvalue) and 0.1 for the bottom, with a linear interpolation between 1 and 0.1 in the remaining directions.

After 100 timesteps, we make rank-$k$ updates to the matrix (in the directions corresponding to the top $k$ eigenvalues). That is,
\[V_t = V_{t-1} + \sum_{j\leq k} u_ju_j^\intercal , \]
where $u_1$ is the eigenvector corresponding to the largest eigenvalue, $u_2$ that corresponding to the second largest eigenvalue, and so on. For small growth of the determinant (the left panel) we use $k=1$, for moderate growth (middle panel) we use $k=10$ and for large growth (right panel) we use $k=19=d-1$ (nearly a full-rank update).

We compare Theorem~\ref{thm:sub-psi} to Corollary 4.2 of \citet{whitehouse2023time}, which is tailored explicitly to sub-$\psi_{G,c}$ processes. It states that if $(S_t,V_t)$ is a sub-$\psi_{G,c}$ process with $V_t\succeq \rho I_d$ for all $t$, then, with probability $1-\delta$, for all $t$,

\begin{align*}
    \|S_t\|_{V_t^{-1}} \leq \frac{1}{1-\eps} \sqrt{4(M_1(t) + M_2(t) + M_3(t))} + \frac{c\eta}{\sqrt{\gamma_{\min}(V_t)}}(M_1(t) + M_2(t) + M_3(t)),
\end{align*}
where
\begin{align*}
    M_1(t) &= B\log\left(A\log_\eta\left(\frac{\gamma_{\max}(V_t)}{\rho}\right)\right), \\
    M_2(t) &= \log\left(\frac{1}{\delta(1 - 1/\beta)}\right), \\
    M_3(t) &= (d+1)\log\left(\frac{\beta \sqrt{\kappa(V_t)}}{\eps}\right).
\end{align*}
Here $\eta$ is the ratio defining their stitching epochs, and $A$ and $B$ are any constants for which their stitching function $h$ obeys $h(k)\leq Ak^B$ for all $k\geq1$. We take $\beta = 2$, $\eta=2$, $\eps=1/2$, $A=7$ and $B=2$, as suggested in their paper, and $\rho=1$ (recall $U_0=I_d$, so $V_t\succeq I_d$ throughout). The values $A=7$ and $B=2$ correspond to the stitching function $h(k) = \zeta(2)(k+1)^2$, the same one we use in the proof of Theorem~\ref{thm:convex_conjugate_bound}. In all plots, we omit the first 100 time steps because our bound blows up when $\gamma_{\min}(V_t)$ is too small. As we say in the main paper, this is a drawback of our bound, and one from which the bound of \citet{whitehouse2023time} does not suffer.

\paragraph{Figure~\ref{fig:emp-bern}.} We take $d=20$ and $\delta=0.05$, and run for $10^5$ steps. For $k\in\{2,10,19\}$, the observations are drawn independently and identically: we draw the first $k$ components of $X_t$ independently and uniformly from $[0,1]$, set the remaining $d-k$ components to zero, and rescale the vector so that $\|X_t\|_2 = 1/2$. The conditional mean $\E[X_t|\calF_{t-1}]$ is therefore a constant $\mu$, as Corollary~\ref{cor:empirical-bernstein-stitched} requires, and every $X_t$ is supported on the first $k$ coordinates. No separate exploration phase is needed to make the normalizer invertible, since $\Gamma+U_0'=4I_d$ already ensures $W_t + U_0' + \Gamma \succeq 4I_d$. We take $\muhat_t$ to be the empirical mean of the first $t$ observations, i.e., $\muhat_t = t^{-1}\sum_{k\leq t} X_k$, with $\muhat_0 = 0$. The random seed is fixed at 42, and, as in Figure~\ref{fig:subgamma}, we plot from $t=100$ onwards. The bound plotted, Corollary~\ref{cor:empirical-bernstein-stitched}, involves no free parameter.

As per the discussion after Corollary~\ref{cor:empirical-bernstein}, when instantiating our bound we split $U_0$ into $\Gamma$ and $U_0'$. For these experiments we take $U_0 = 4I_d$, $\Gamma = 3.5 I_d$ and $U_0' = 0.5I_d$, so that $\Gamma + U_0' = U_0$. We then apply the result with the process $(S_t,\Gamma + W_t)$, for $W_t$ as in~\eqref{eq:eb-process}, and the fixed matrix $U_0'$ in place of $U_0$. Both our bound and that of \citet{whitehouse2023time} are therefore bounds on $\|S_\tau\|_{M_\tau^{-1}}$ for the same matrix $M_\tau = W_\tau + 4I_d$.

We compare Corollary~\ref{cor:empirical-bernstein-stitched} with the empirical Bernstein bound of \citet[Theorem 5.4]{whitehouse2023time}, which is given by~\eqref{eq:whitehouse-eb-bound} and which we evaluate with the same hyperparameters $\eps=1/2$, $\alpha=2$, $\beta=2$ and $\rho=\gamma_{\min}(U_0)=4$. 